\documentclass{amsart}
\usepackage[all]{xy}
\usepackage{amsmath}
\usepackage{amssymb}
\usepackage{amsthm}
\usepackage{amscd}

\newcommand{\R}{\mathcal{R}}
\newcommand{\M}{\mathcal{M}}

\newcommand{\E}{\mathcal{E}}

\newcommand{\A}{\mathcal{A}}

\newcommand{\s}{\mathcal{S}}

\newcommand{\Z}{\Bbb Z}

\newcommand{\CC}{\Bbb C}

\newtheorem{theorem}{Theorem}
\newtheorem{definition}[theorem]{Definition}
\newtheorem{proposition}[theorem]{Proposition}
\newtheorem{cor}[theorem]{Corollary}
\newtheorem{lemma}[theorem]{Lemma}
\newtheorem{problem}[theorem]{Problem}
\newtheorem{remark}[theorem]{Remark}

\numberwithin{theorem}{section}

\begin{document}
\title{On Higher-Dimensional Oscillation in Ergodic Theory}
%Enter your title between curly braces
\author{Ben Krause}
\address{UCLA Math Sciences Building\\
         Los Angeles,CA 90095-1555}
\email{benkrause23@math.ucla.edu}
\date{\today}
\maketitle

\begin{abstract}
We extend the results of \cite{J1} concerning higher-dimensional oscillation in ergodic theory in a variety of ways.
We do so by transference to the integer lattice \cite{C1}, where we employ technique from (discrete) harmonic analysis.
\end{abstract}

\section{Introduction}
Let $(X,\Sigma,\mu)$ be a non-atomic probability space, equipped with $T$ a measure preserving $\Z^d$-action
\[ T_y f(x):= f(T_{-y} x). \]

For $\{E_i\} \subset \Z^d$, define the averaging operators
\[ M_i f(x):= \frac{1}{|E_i|} \sum_{y\in E_i} (\tau_yf)( x);\]
the classical pointwise ergodic theorem of Birkhoff says that if $E_i= [0,i) \subset \Z^{1}$, then the one-dimensional averages
$\{ M_i f(x) \}$
converge pointwise $\mu$-almost everywhere for $f \in L^p(X,\Sigma,\mu), \ 1\leq p<\infty$. A standard proof proceeds by way of a density argument, where the key quantitative estimate
is that the one-dimensional maximal function
\[ f^{*,1}:=\sup_i |M_i f| \]
is of weak-type $(1,1)$, and strong-type $(p,p)$, $1 < p \leq \infty$. Explicitly, there exist absolute constants, $C_p, \ 1 \leq p \leq \infty$ so that
\[ \aligned
\lambda \mu \left( \left\{ f^{*,1} > \lambda \right\} \right) &\leq C_1 \|f\|_1, \text{ for $\lambda \geq 0$, and } \\
\left\| f^{*,1}  \right\|_p &\leq C_p \|f\|_p, \ 1< p \leq \infty. \endaligned\]
This result was generalized to higher dimensions by Wiener; a direct consequence of \cite[Theorem $II'$]{WI} is that if $E_i \subset \Z^d$ are a nested, increasing sequence of cubes which contain the origin, then the $d$-dimensional averages $\{ M_i f(x) \}$
converge pointwise $\mu$-almost everywhere, and the $d$-dimensional maximal function
\[ f^{*,d}:= \sup_i |M_i f| \]
is similarly weak-type $(1,1)$ and strong-type $(p,p), \ 1< p \leq \infty$.

A modern path to Wiener's result is through the transference principle of Calder\'{o}n \cite{C1}, which allows one to conduct the study of ergodic averages on the lattice, $\Z^d$:
in particular, it is enough to consider the discrete convolution operators $A_i,$ defined by
\[ A_i f(n) := \frac{1}{|E_i|} \sum_{m \in E_i} f(n-m).\]
An advantage to this perspective is that real analytic methods -- covering lemmas, Fourier analysis and further orthogonality techniques -- can be brought to bear in studying more general regions $\{E_i\} \subset \Z^d$ and pertaining the averaging operators $\{A_i\}$. Indeed, provided the collection of sets $\{ E_i \} \subset \Z^d$ being studied share some qualitative properties with an increasing collection of cubes \cite[\S\S 1-2]{S}:
\begin{itemize}
\item A one-parameter structure; and
\item Some geometric ``smoothness''
\end{itemize}
the maximal functions on the lattice,
\[ \sup_i |A_i f|\]
are of weak-type $(1,1)$ and strong-type $(p,p), 1< p \leq \infty$.

Obtaining pointwise convergence results of $\{ A_i f(n) \}$ for more exotic $\{E_i \} \subset \Z^d$ -- those for which the above two properties are relaxed, or absent --
does not necessarily follow from quantitative estimates on an appropriate maximal function, since the dense-subclass result is often unavailable in this setting.
Perhaps the most famous instance of this difficulty arose in the study of averages along the squares, i.e.\
\[ E_i := \{1,4,9,\dots, i^2\} \subset \Z^1.\]
Indeed, to prove pointwise convergence of the ergodic averages of $L^2$-functions along the squares, Bourgain \cite{B1} showed that for any lacunarily increasing sequence $\{i_k\}$, the
\emph{oscillation operator},
\[ \mathcal{O}f := \left( \sum_k \sup_{i_k < i \leq i_{k+1}} | A_{i_k} f - A_i f|^2 \right)^{1/2} \]
was of strong-type $(2,2)$. The $L^2$ result then anchored a density argument through which he was able to extend his result to all $p>1$.

In proving this result, Bourgain made use of the \emph{$s$-variation operators}, $s > 2$, classically used in probability theory to gain quantitative information on rates of convergence.
\footnote{One representative example which in fact appears in Bourgain's argument is L\'{e}pingle's inequality for martingales \cite{LE}.}
More precisely, Bourgain proved that, in the special case $E_i = [0,i) \subset \Z^1$, the $s$-variation operators
\[ \mathcal{V}^sf = \mathcal{V}^s_{\{E_i\}} f := \sup_{ i_1 < i_2 < \dots < i_N } \left( \sum_{k=1}^N | A_{i_k} f - A_{i_{k+1}} f|^s \right)^{1/s} \]
were of strong-type $(2,2)$ \cite[Corollary 3.26]{B1}.

These variation operators are more difficult to control than the maximal function $\sup_i |A_i f|$: for any $j$, one may pointwise dominate
\[ \sup_i |A_i f| \leq \mathcal{V}^{\infty} f + |A_{j}f| \leq \mathcal{V}^{s} f + |A_{j}f|,\]
where $ s < \infty$ is arbitrary. On the other hand, variational (or oscillation) estimates are a powerful tool for proving pointwise convergence when a density argument seems unavailable.

Since Bourgain's celebrated result, establishing variational estimates for families of averaging operators has been the focus of much research in ergodic theory. \footnote{Variation estimates have also been studied from a harmonic analysis perspective in the context of truncations of singular integral operators. We refer the reader to \cite{JSW} for further discussion.}

A fundamental paper in this direction is due to Jones et.\ al.\ \cite{J}, where it is shown that the one-dimensional variation operators $\{ \mathcal{V}^s \}$ are of weak-type $(1,1)$ and strong-type $(p,p), \ 1< p < \infty$.
In other words, the variation operators enjoy the same boundedness properties as their associated maximal function. The argument in \cite{J} proceeded by first controlling an oscillation operator adapted to $\{E_i\} = \{[0,i)\}$ and then using martingale-style techniques from probability theory; the approach to the oscillation operator itself, however, was driven by Fourier-based orthogonality arguments, initially found in \cite{B1}.

A subsequent paper of Jones, Rosenblatt, and Wierdl, \cite{J1}, refined the orthogonality methods used in \cite{J} by eliminating the Fourier-analytic techniques, and more closely following (dyadic) martingale-style arguments.
In so doing, the authors were able to establish analogous results in higher-dimensional settings for functions in the ``low''-$L^p$ regime, $1 < p \leq 2$.

To continue our discussion, we briefly introduce two representative operators studied in \cite{J1}: the ``pointwise'' square functions. In our current context, the significance of these operators is that establishing $L^p$ bounds leads directly to bounds on the corresponding oscillation and variation operators \cite[\S 1]{J1}:

Suppose that $\{E_i\}$ are cubes of side-length $i$, so that for $2^{k-1} \leq i < 2^k$,
\[ E_i \subset H_k:=\prod_{i=1}^d [-2^k,2^k).\]
In other words, the $\{E_i\}$ are displaced from the origin by an amount comparable to, or less than, their side lengths. Then, with
\[ A^kf(x) := \frac{1}{|H_k|} \sum_{m \in H_k} f(n-m), \]
we define the \emph{long} square function (on the integers) with respect the collection $(E_i)$ as
\[ \s^L_{\{E_i\}}f(x) := \left( \sum_k \left( \sup_{2^{k-1} \leq i < 2^k} | A_i f(x) - A^k f(x)|^2 \right) \right)^{1/2},\]
and the \emph{short} square function as
\[ \s^S_{\{E_i\}}f(x) := \left( \sum_k \left( \sup_{2^{k-1} \leq t_i < 2^k \text{ increasing}} | A_{t_i} f(x) - A_{t_{i+1}} f(x)|^2 \right) \right)^{1/2}.\]

A direct consequence of \cite[Theorems $A'$,$B'$]{J1} is the following:
\begin{theorem}
There exist absolute constants $C_1, C_p$, $1<p\leq 2$ so that
\[ \aligned
\left\| \s^*_{\{E_i\}}f(x) \right\|_{L^p(\Z^d)} &\leq C_p \| f \|_{L^p(\Z^d)} \\
\left\| \s^*_{\{E_i\}}f(x) \right\|_{L^{1,\infty}(\Z^d)} &\leq C_1 \| f \|_{L^1(\Z^d)}, \endaligned \]
where $* = L,S$.
\end{theorem}
The two operators, $\s^L,\s^S$, both measure the scales and locations where $f$ oscillates, i.e.\ differs from constant functions.
The content of these results -- and indeed the ideas that drive their proofs -- is that there is sufficient orthogonality between the various scales to ensure that, in an average sense, even a \emph{pointwise} measurement of oscillation is controllable by the size of the initial function. We refer the reader to $\S 2$, or to \cite[\S\S 3-4]{J1} for further discussion.

A key insight in \cite{J1} is that just as in the case of the maximal function, $\sup_i |A_i f|$, the $\{E_i\}$ do not actually need to be cubes for the pertaining square functions to enjoy the above control. Indeed, the above theorem holds provided the collection of sets $\{ E_i \}$ being studied shared the same qualitative properties with a nested collection of cubes as in the case of the maximal function: \footnote{For a more precise statement of these results, we refer the reader to \cite[\S 1]{J1}, or to \S 2 below;
for an excellent treatment of continuous analogues of the square functions introduced above, we refer the reader to the discussion of the \emph{intrinsic square function}, found in e.g.\ \cite[\S 6]{W}.}
\begin{itemize}
\item A one-parameter structure; and
\item Some geometric ``smoothness.''
\end{itemize}

In light of the analogous approaches to studying maximal functions and square functions in our current context, the following informal question seems natural:

\begin{quote}
To what (further) extent do the boundedness properties of the square functions under our consideration parallel those of the maximal functions?
\end{quote}

More precisely, we organize our study of the operators introduced in \cite{J1} according to the following aims:

First, we seek to extend our control of the square functions under the assumptions outlined \cite{J1}. An immediate concern is the behavior of the square function in the ``high"-$L^p(\Z^d)$ regime, which we investigate by using sharp-function techniques from harmonic analysis. We prove:

\begin{theorem}\label{hi}
If $\A= \{E_i\}$ is \emph{regular}, then the long and short square functions $\s_{\A}^{L},\s_{\A}^{S}$, are $L^p(\Z^d)$-bounded, $2<p<\infty$.
\end{theorem}
We refer the reader to $\S 2$ for the precise definition of \emph{regularity}; informally, regular sequences $\{E_i\}$ share the above-mentioned qualitative similarities to nested cubes.

This result leads directly to new control over jump and variation inequalities:

For $\lambda > 0$, we define, as in \cite{J1},
\[ J( (A_i),\lambda )= J( (A_i(f,x),\lambda ) ) \]
as the largest $N$ for which there are increasing indices $(i_j)$ with
\[ |(A_{i_j} - A_{i_{j+1}})f(x)| > \lambda, \ 1 \leq j < N.\]

Then, arguing as in \cite{J1}, \S 1, under the assumption of regularity we have the following:
\begin{cor}
For $2<p< \infty$ there exist absolute constants $C_p$ so that for any $\lambda >0$,
\[ \left\| \lambda \cdot  J( (A_i),\lambda )^{1/2} \right\|_{L^p(\Z^d)} \leq C_p \|f\|_{L^p(\Z^d)}.\]
Moreover, for any $s > 2$,
\[ \left\| \mathcal{V}^s f \right\|_{L^p(\Z^d)} \leq C_p \|f \|_{L^p(\Z^d)}. \]
\end{cor}
We remark that by transference this implies the corresponding result for dynamical systems.
\begin{cor}
For $2<p< \infty$ there exist absolute constants $C_p$ so that for any $\lambda >0$,
\[ \left\| \lambda \cdot  J( (M_i),\lambda )^{1/2} \right\|_{L^p(X,\mu)} \leq C_p \|f\|_{L^p(X,\mu)}.\]
Moreover, for any $s > 2$,
\[ \left\| \mathcal{V}^s f \right\|_{L^p(X,\mu)} \leq C_p \|f \|_{L^p(X,\mu)}. \]
\end{cor}

%Moreover, we are able to establish the reverse inequalities, which strengthen the above-mentioned orthogonality relationships between martingale increments and pointwise measurements of oscillation.

%\begin{proposition}\label{rev}
%For $1<p<\infty$,
%\[ \|f \|_p \lesssim \| \s_{\A}^{*} f\|_p,\]
%$* = L,S$.
%\end{proposition}

To deepen the connection between square and maximal function, we also consider the behavior of square functions on the discrete $A_p(\Z^d)$-weighted classes. We are additionally motivated in this regard by the weighted estimates for one-parameter actions studied in \cite{LT}, \cite{GT}, and \cite{MT}, and by the weighted theory of ``rough'' singular integrals which satisfy the so-called H\"{o}rmander conditions \cite{LRT}, \cite{LMPR}.

We first recall a standard characterization of $A_p$ weights; we refer the reader to \cite[\S 7]{D1} or to \cite[\S 5]{S} for a more comprehensive treatment. We begin by recalling the discrete uncentered cubic Hardy-Littlewood maximal function.

\begin{definition} For $f : \Z^d \to \CC$, we define the discrete uncentered cubic Hardy-Littlewood maximal function:
\[ M_{HL} f(x) := \sup_{Q \ni x} \frac{1}{|Q|} \int_Q |f|,\]
where the supremum is taken over all cubes $Q$ which contain the point $x$.
\end{definition}

With this operator in hand, we are free to define and introduce some properties of $A_p(\Z^d)$ weights.

\begin{definition}
For $1 < p < \infty$, a positive function $w \in A_p(\Z^d)$ is a \emph{discrete $A_p$ weight} if $M_{HL}$ is bounded on $L^p(\Z^d)$. Explicitly, $w \in A_p(\Z^d)$ if and only if there exists an absolute $C_p > 0$ so that
\[ \aligned
\int_{\Z^d} |M_{HL} f|^p w &\leq C_p^p \int_{\Z^d} |f|^p w \; \text{ or } \\
\|M_{HL} f\|_{L^p(\Z^d, w)} &\leq C_p \|f\|_{L^p(\Z^d, w)}. \endaligned \]

A positive function $w \in A_1(\Z^d)$ is an \emph{$A_1$ weight} if there exists a constant $C=C(w)$ so that $M_{HL}w \leq C w$.

We say $w \in A_{\infty}(\Z^d)$ is an \emph{$A_{\infty}$ weight} if $w \in A_p(\Z^d)$ for some finite $p$.
\end{definition}

We isolate two important properties of $A_1(\Z^d)$ and $A_\infty(\Z^d)$ weights; the proofs of these facts can be found in e.g.\ \cite[\S 5]{S}:
\begin{itemize}
\item $A_1$ weights $w \in A_1$ satisfy $\frac{w(Q)}{|Q|} \leq B \inf_{x \in Q} w(x)$ for any cube $Q$, and an absolute $B= B(w)$; and
\item $A_\infty$ weights are automatically doubling: $w \in A_\infty$ automatically satisfy $w(2Q) \leq C w(Q)$ for any cube $Q$, and an absolute $C=C(w)$. Here $2Q$ denotes the cube concentric with $Q$, but with twice the side length.
\end{itemize}

We say that a collection of sets $\{E_i\}$ is \emph{cubic} if the maximal function $\sup_i |A_i f|$ is (up to constant factors) pointwise dominated by $M_{HL} f$. Under this natural assumption, we have the following theorem.

\begin{theorem}\label{weight}
If $\A = \{E_i\}$ is \emph{cubic}, then for
both the long and short square functions $\s_{\A}^{*}, \  *=L,S$ there exist absolute constants $C_p$ so that
\[ \int_{\Z^d} |\s_{\A}^{*}f|^p w \leq C_p \ \int_{\Z^d} |f|^p w \]
for any $A_p$-weight $w$, $1\leq p < \infty$.
\end{theorem}
%A definition of ``\emph{cubicity}'' can be found in $\S 4$; as the name suggests, ``cubic'' sequences more closely resemble nested sequences of cubes than regular sequences.

Again arguing as in \cite[\S 1]{J1}, this implies the following corollary for cubic families:
\begin{cor}
For $1<p< \infty$ there exist absolute constants $C_p$ so that for any $\lambda >0$,
\[ \left\| \lambda \cdot  J( (A_i),\lambda )^{1/2} \right\|_{L^p(\Z^d, w)} \leq C_p \|f\|_{L^p(\Z^d, w)}.\]
Moreover, for any $s > 2$,
\[ \left\| \mathcal{V}^s f \right\|_{L^p(\Z^d, w)} \leq C_p \|f \|_{L^p(\Z^d, w)}. \]
\end{cor}

Of course, by specializing to the weight $v \equiv 1$, we recover the primary results of \cite{J1}.

%An immediate question concerns the behavior of our square functions in the high-$l^p$ range.
%We include an elementary approach to studying high-$l^p$ square functions in an appendix.
%Motivated by results in dyadic harmonic analysis, we use stopping time and ``good-$\lambda$'' techniques introduced by BURKHOLDER to prove stronger weighted inequalities.

Finally, we investigate the behavior of our square functions when the crucial smoothness assumption is relaxed. We focus our efforts in this regard on the following problem.

%A preliminary question concerns the behavior of square functions in the high-$l^p$ range; by arguing as in \cite{DR}, we are able to quickly prove the $l^p$-boundedness of almost all square functions considered in \cite{J1}.
%Connections between square function and function are much stronger.

%BURKHOLDER
%...survey of
%...useful presentation in \cite{W}

%In this regard, it seems quite natural that the above two conditions are actually enough to provide stronger connections between the square functions studied in \cite{J1} and the dyadic square-functions. An important tool introduced by BURKHOLDER in the study of (martingale) square functions are relative distributional, or ``good-$\lambda$, inequalities.

%In particular, we extend the results of \cite{J1} to the ``high-$l^p$'' setting $2 < p < \infty$,
%and prove weighted estimates as well. We achieve these extensions by combining the orthogonality technique developed in \cite{J1} with stopping-time arguments used in dyadic harmonic analysis.

%Secondly, in $\S 5$, we consider what happens when the ``smoothness'' assumption is relaxed:
\begin{problem}[\cite{J1}, Problem 7.5]
For each collection of nested rectangles $\{E_i\} \subset \Z^d$, is it true that for each $\phi \in L^1(X,\Sigma,\mu)$,
\[ \left( \sum_i |(M_i - M_{i+1})\phi|^2 \right)^{1/2} < \infty \]
$\mu$-almost everywhere?
\end{problem}
Certainly, this result would be implied by a weak-type bound
\[ \left\| \left(\sum_i |(M_i - M_{i+1})\phi|^2\right)^{1/2} \right\|_{L^{1,\infty}(X)} \lesssim \| \phi \|_{L^1(X)}. \]
In fact, as in shown in the Appendix \S 6 below, in many cases this weak-type bound is \emph{necessary} for convergence to occur. We therefore focus our attention on the following slightly more general

\begin{problem}[\cite{J1}, Problem 7.5 -- Working Version]\label{Rec}
For each collection of nested rectangles $\{E_i\} \subset \Z^d$, does there exist a bound
\[ \left\| \left(\sum_i |(A_i - A_{i+1})f|^2\right)^{1/2}  \right\|_{L^{1,\infty}(\Z^d)} \lesssim \|f\|_{L^1(\Z^d)}?\]
\end{problem}
Though this problem remains out of reach in its fullest generality, we are able to answer the problem affirmatively under a lacunarity assumption on collection $\{E_i\} \subset \Z^d$.

\begin{proposition}
For each collection of nested rectangles $\{E_i\} \subset \Z^d$ with \emph{dyadic} side lengths,
\[ \left\| \left(\sum_i |(A_i - A_{i+1})f|^2\right)^{1/2} \right\|_{L^{1,\infty}(\Z^d)} \lesssim \|f\|_{L^1(\Z^d)}.\]
\end{proposition}

\bigskip

The structure of the paper is as follows:

In $\S 2$ we introduce relevant definitions, and present a few reductions which will be used throughout;

In $\S 3$, we study our square functions' behavior in the high-$L^p(\Z^d)$ regime;

In $\S 4$, we prove weighted estimates; and

In $\S 5$ we relax the smoothness assumptions of \cite{J1}, and discuss Problem \ref{Rec}.

Our appendix, $\S 6$, contains a weak-type principle for square functions in the spirit of \cite{S1}.

%in $\S 5$ we prove that a weak-type bound exists for a ``global'' variant of rectangular square functions.
%We also highlight difficulties present in the ``local'' variant, and propose a modified problem.

\subsection{Acknowledgements}
The author would like to thank Benjamin Hayes and Michael Lacey for helpful conversations, and his advisor, Terence Tao, for his great patience and support.

\subsection{Notation}
We shall, whenever possible, maintain the notation introduced in \cite{J1}.

For a set $E \subset \Z^d$, we use $|E|$ to denote the cardinality of the set $E$.

We let $1_E$ denote the indicator function of the set $E$, i.e.\
\[ 1_E(x) =
\begin{cases} 1 &\mbox{if } x \in E \\
0 & \mbox{if } x \notin E. \end{cases}\]
We let $\chi_E := \frac{1}{|E|} \cdot 1_E$ denote the $L^1(\Z^d)$-normalized indicator function.

For a function $f$ defined on $\Z^d$, our convention will be to let
\[ \int f \]
denote the summation $\sum_{n \in \Z^d} f(n)$. Accordingly, we will use
\[ \| f \|_{L^p(\Z^d)} \]
to denote the $L^p(\Z^d)$-norm, $(\sum_{n \in \Z^d} |f(n)|^p)^{1/p}$, with the obvious modification for $p=\infty$.

When integrating over other spaces, we will include the domain and measures.

We will make use of the modified Vinogradov notation. We use $X \lesssim Y$, or $Y \gtrsim X$ to denote the estimate $X \leq CY$ for an absolute constant $C$. If we need $C$ to depend on a parameter,
we shall indicate this by subscripts, thus for instance $X \lesssim_p Y$ denotes
the estimate $X \leq C_p Y$ for some $C_p$ depending on $p$. We use $X \approx Y$ as
shorthand for $X \lesssim Y \lesssim X$.

\section{Preliminaries}
Let $\R = 0 \subset R_1 \subset R_2 \subset \dots $ denote a nested sequence of dyadic rectangles inside $\Z^d$, i.e.\
\[ R_k = \prod_{a=1}^d [0, 2^{a(k)});\]
here $\{ a(k) \}$ are non-decreasing, and $a'(k) < a'(k+1)$ for at least one value $1 \leq a' \leq d$.

We also define the ``symmetric'' rectangles
\[ H_k = \prod_{a=1}^d [-2^{a(k)}, 2^{a(k)}).\]
%and their $8$-fold dilates:
%\[ H'_k = \prod_{i=1}^d [-2^{i(k)+3}, 2^{i(k)+3}).\]

For $k \geq 0$, we let $\sigma_k = \sigma_k(\R)$ denote the $\sigma$-algebra generated by $R_k$, i.e.\ the $\sigma$-algebra with atoms
\[ \prod_{a=1}^d 2^{a(k)}[m, m+1).\]
Henceforth, we let $\E_k$ denote the expectation with respect to $\sigma_k$, $\E_0$ the identity operator,
\[ \Delta_0 = \E_0 \ \text{ and } \
 \Delta_k:= \E_k - \E_{k-1}, \ k \geq 1 \]
denote the martingale differences.

Closely connected to our family of rectangles, $\R$, are the collections of sets $\A = \{\A_k\}$, whose elements have controlled
\begin{enumerate}
\item Spatial location: for every $E \in \A_k$, $E \subset H_k$;
\item Size: for each $E \in \A_k$ there exists some $c > 0$ so that $c|H_k| \leq |E|$;
\item (Internal) Smoothness: for each $E \in \A_{k+l}$, for $l \geq 0$
\[ \frac{ | \{ x: \partial E \cap (H_{k} - x) \neq \emptyset \} | }{ |H_{k+l}| } =: \frac{B(E,H_k)}{|H_{k+l}|} \leq \iota(l)^2, \]
with $\sum_l \iota(l) < \infty$; and finally
\item Eccentricity:
\[ \frac{ B(H_{k+l}, H_{k}) }{ |H_{k+l}| } \leq \varepsilon(l)^2,\]
where $\sum_l \varepsilon(l) < \infty$ as well.
\end{enumerate}
We shall collectively refer to the above four criteria as the \emph{JRW criteria}, and use the notation $\A = \{ \A_k \}$ to denote the various collections.

The third of the above points forces regularity on the boundaries of the $E \in \A_k$, i.e.\ some smoothness on $\chi_E$.

The fourth point is implicitly used in the proofs of the main theorems of \cite{J1}, though not explicitly stated in the summary of \cite[\S 5]{J1}.
We shall replace it with the following equivalent formulation, which we isolate in the form of the following simple

\begin{lemma}
Eccentricity control as above is equivalent to the existence of an $L$ so that
\[ \min_{1 \leq a \leq d} \{a(k+L) - a(k)\} \geq 1 \]
for each $k$, with $H_k = \prod_{a=1}^d [-2^{a(k)}), 2^{a(k)})$.

Moreover, the existence of such an $L$ allows us to take
\[ \varepsilon(l)^2 \lesssim 2^{- \frac{l}{L} },\]
and therefore
\[ \sum_l \varepsilon(l)^\theta < \infty, \text{ for any } \theta > 0.\]
\end{lemma}
We say that two cubes $Q = \prod_{i=1}^d I_i, \ R = \prod_{i=1}^d J_i, |I_i| \geq |J_i|$
with dyadic side-lengths are \emph{$s$-separated} if
\[ \min_{ 1 \leq i \leq d} \left\{ \frac{|J_i|}{|I_i|} \right\} = 2^{-s};\]
the content of the above theorem is that there exists an absolute $L$ so that
$H_{k+L}, H_k$ are 1-separated for each $k$.
\begin{proof}
If $s = \min_{1 \leq i \leq d} \{a(k+l) - a(k)\}$, then
\[ \frac{ B(H_{k+l}, H_{k}) }{ |H_{k+l}| } \approx 2^{-s};\]
if no such $L$ were to exist, then we could find arbitrarily many $l, k=k(l)$ so that
\[ \frac{ B(H_{k+l}, H_{k}) }{ |H_{k+l}| } \gtrsim 1,\]
which would force the sum $\sum_l \varepsilon(l)$ to diverge.
\end{proof}

For our purposes, the (alternate) eccentricity condition guarantees that for any dyadic $2^c , c \geq 1$,
\[ \frac{B(H_{k+l}, 2^c H_k)}{|H_{k+l}|} \lesssim 2^{- \frac{l}{cL} }.
\]

We shall be studying families of sets, $\{E_t\}$, which are \emph{regular} with respect to our collections $\A = \{\A_k\}$:
for each $2^k \leq t \leq t' < 2^{k+1}$,
\[ E_t \subset E_{t'}, \text{ and } E_t, E_{t'} \in \mathcal{A}_k.\]
If in addition $\A = \{\A_k\}$ satisfy the above JRW criteria, we will say that $\A$ itself is \emph{regular}.

%In $\S 5$, we shall relax the Smoothness and Eccentricity assumptions on our families.

With a regular collection $\A = \{\A_k\}$ and $\{E_t\}$ specified,
we will let
\[ \chi_t = \chi^{\A}_t := \frac{1}{|E_t|} 1_{E_t} \]
denote the $L^1(\Z^d)$-normalized indicator function.

We let
\[ A_t g (x) = A^{\A}_t g(x) := \chi_t *g(x) = \frac{1}{|E_t|} \sum_{y \in E_t} g(x-y) \]
denote the convolution operator with kernel $\chi_t$.
We introduce the maximal function associated to $\A$
\[ Mf = M_{\A} f:= \sup_k \chi_{5H_k} *|f|(x), \]
which dominates $\sup_t A_t |f|$, and satisfies a weak-type $(1,1)$ inequality by the nesting properties of the $\{H_k\}$. (cf. e.g.\ \cite[Lemma 5.3]{SW1}).

With $\{E_t\}$ regular with respect to $\A$, we define the long square function
\[ \s_{\A}^L f:= \left(\sum_k |\s_{\A,k}^{L} f|^2\right)^{1/2} := \left(\sum_k \sup_{2^k \leq t < 2^{k+1}} | A_t f - \E_kf|^2\right)^{1/2},\]
and short square function
\[ \s_{\A}^S f:= \left(\sum_k |\s_{\A,k}^{S} f|^2\right)^{1/2} := \left(\sum_k \sup_{t_i \text{ increasing}} \sum_{2^k \leq t_i < 2^{k+1}} | A_{t_i} f - A_{t_{i+1}} f|^2\right)^{1/2}.\]
Often, we will suppress the subscript $\A$.

We briefly remark that that for $2^k \leq t < 2^{k+1}$, since
\[
|A_tf|(x) \lesssim \chi_{H_k} *|f|(x) \ \text{ and } |\E_k f| (x) \lesssim \chi_{H_k} *|f|(x), \]
we may control $\s_{k}^{L} f \lesssim \chi_{H_k} *|f|$.
We have similar control over $\s_{k}^{S} f$. Using the domination of the $l^2$ norm by the $l^1$ norm, and then the triangle inequality, we may bound
\[ \aligned
\s^S_k f &:= \sup_{t_i \ \text{increasing}} \left(\sum_{2^k \leq  t_i  < 2^{k+1} } |(A_{t_i} - A_{t_{i+1}}) f|^2 \right)^{1/2}(x) \\
&\leq \sum_{2^k \leq t < 2^{k+1} } |(A_t - A_{t+1}) f|(x) \\
&\leq  \sum_{2^k \leq  t < 2^{k+1} } \left( \left(\frac{1}{|E_t|} - \frac{1}{|E_{t+1}|} \right) 1_{E_t} * |f| \right) +
\sum_{2^k \leq  t < 2^{k+1} } \left(\frac{1}{|E_{t+1}|} 1_{E_{t+1} \smallsetminus E_t} * |f| \right) \\
&\lesssim \chi_{H_k}*|f|, \endaligned \]
where we used the size control $c|H_k| \leq |E_t|$ in the final inequality. In particular, we may control each
\[ \s^*_k f \lesssim Mf.\]

We also introduce the larger, shifted square functions:
\[ \tilde{\s}_{\A}^{*} f(x):= \left(\sum_k |\tilde{\s}_{\A,k}^{*} f(x)|^2 \right)^{1/2} := \left(\sum_k |\sup_{v \in H_k} \s_{\A,k}^{*}f(x+v) |^2 \right)^{1/2},\]
$*= L, S$.
The additional suprema affords the shifted square functions a useful degree of smoothness:

\begin{lemma}\label{smooth}
With $x_Q$ denoting the center of each $Q \in \sigma_k$,
$\tilde{\s}_k^*f$ is pointwise dominated, and $L^p(\Z^d)$-comparable $1\leq p \leq \infty$, to a function which is constant on $Q \in \sigma_k$:
\[ \s_{D,k}^{*} f:= \sum_{Q \in \sigma_k} \tilde{\s}_{k}^{*} f(x_Q) 1_{3Q}.\]
In particular,
\[ \left\| \tilde{\s}_{k}^{*} f(x) \right\|_{L^p(\Z^d)} \approx_p \left\| \s_{D,k}^{*} f \right\|_{L^p(\Z^d)}.
 \]
\end{lemma}
\begin{proof}
If $Q_i \in \sigma_k$ lie in $3Q$, then for any $x\in Q$, we may bound
\[ \tilde{\s}_{k}^{*} f(x) \leq \sum_i \tilde{\s}_{k}^{*} f(x_{Q_i}).\]
Indeed, if $v \in H_k$ is such that
\[ \tilde{\s}_{k}^{*} f(x) = {\s}_{k}^{*} f(x + v),\]
then we may write $x+v = x_{Q_i} + v_{Q_i}$ for some center $x_{Q_i}$ of a neighboring cube $Q_i \in 3Q$, and some (possibly different) $v_{Q_i} \in H_k$.
Summing over all $Q \in \sigma_k$ therefore yields a pointwise majorization
\[ \tilde{\s}_{k}^{*} f(x) \lesssim \sum_Q \tilde{\s}_{k}^{*} f(x_Q)1_{3Q}.\]

On the other hand, if $v_Q \in H_k$ is such that
\[ \tilde{\s}_{k}^{*} f(x_Q) = \s_{k}^{*} f(x_Q+v_Q),\]
then on the set
\[ X(Q) = \{y \in Q : y+v = x_Q +v_Q \text{ for some } v\in H_k \} = Q \cap \{x_Q + v_Q - H_k \}, \]
which has measure $|X(Q)| \gtrsim_d |Q|$,
we may bound
\[ \tilde{\s}_{k}^{*} f(x_Q) \leq \inf_{y \in X(Q)} \tilde{\s}_{k}^{*} f(y),\]
by definition of the set $X(Q)$.
We therefore have a similar pointwise inequality for $p=\infty$, while for $1\leq p < \infty$, using the finite overlap of $\{3Q\}$ to estimate pointwise
\[
\left| \sum_Q \tilde{\s}_{k}^{*} f(x_Q)1_{3Q} \right|^p \lesssim_p \sum_Q \tilde{\s}_{k}^{*} f(x_Q)^p 1_{3Q} ,
\]
we estimate
\[ \aligned
\left\| \sum_Q \tilde{\s}_{k}^{*} f(x_Q)1_{3Q} \right\|_{L^p(\Z^d)}^p &\lesssim \sum_Q |\tilde{\s}_{k}^{*} f(x_Q)|^p |3Q| \\
&\lesssim \sum_Q |\tilde{\s}_{k}^{*} f(x_Q)|^p |X(Q)| \\
&\leq \sum_Q \int_{X(Q)} |\tilde{\s}_{k}^{*}f(y)|^p \\
&\leq \sum_Q \int_{Q} |\tilde{\s}_{k}^{*}f(y)|^p \\
&= \left\| \tilde{\s}_{k}^{*} f \right\|_{L^p(\Z^d)}^p. \endaligned \]
\end{proof}

For future use, we record the following additional
\begin{lemma}
For any $v \in A_\infty$,
\[ \s_{D,k}^{*} f:= \sum_{Q \in \sigma_k} \tilde{\s}_{k}^{*} f(x_Q) 1_{3Q} \]
is $L^p(v)$-comparable, $1 \leq p < \infty$ to
\[ \s_{d,k}^{*} f:= \sum_{Q \in \sigma_k} \tilde{\s}_{k}^{*} f(x_Q) 1_{Q}.\]
\end{lemma}
\begin{proof}
Using again the bounded overlap of $\{ 3Q \}$, one estimates
\[ \aligned
\| \s_{D,k}^* f \|_{L^p(v)}^p &\lesssim \int \sum_Q |\tilde{\s}_{k}^{*} f(x_Q) |^p 1_{3Q} v \\
&\lesssim \sum_Q |\tilde{\s}_{k}^{*} f(x_Q)|^p v(3Q) \\
&\lesssim \sum_Q |\tilde{\s}_{k}^{*} f(x_Q)|^p v(Q) \\
&= \| \s_{d,k}^* f \|_{L^p(v)}^p, \endaligned \]
where we used the doubling nature of $v \in A_\infty$ in passing to the third line.
\end{proof}

Though -- as we shall see -- more is true, for now we shall only need that our shifted square functions inherit $L^2$-boundedness from their centered associates:
\begin{proposition}\label{L2}
$\left\| \tilde{\s}_{\A}^*f \right\|_{L^2(\Z^d)} \lesssim \|f\|_{L^2(\Z^d)}$.
\end{proposition}
The argument here is very similar to the arguments of \cite[\S 3]{J1}.
%Informally, $\tilde{\s}_k^{*}$ measures the locations where $f$ differs from being constant at $\approx \sigma_k$-scale
%\begin{itemize}
%\item $\tilde{\s}_k^{*}f(x) \equiv \tilde{\s}_k^{*}(f1_{x + 2H_k})(x);$
%\item $\tilde{\s}_k^*(1_{x + 2H_k})(x) = 0$; and
%\item $\tilde{\s}_{k}^{*} f(x) \approx \sum_{Q \in \sigma_n} \tilde{\s}_{k}^{*} f(x_Q) 1_Q$ (Lemma \ref{smooth}).
%\end{itemize}
The qualitative similarities between $\tilde{\s}_k^{*}$ and the projection operators $\Delta_k = \E_k - \E_{k-1}$ -- informally,
both measure the locations where $f$ differs from being constant at $\approx \sigma_k$-scale -- motivate the following orthogonality approach:
\begin{proof}
With
\[ f = \sum_{k \geq 0} \Delta_k(f) =: \sum_{k \geq 0} d_k \]
we majorize
\[ \tilde{\s}^*f = (\sum_n |\tilde{\s}_n^* ( \sum_k d_k)|^2)^{1/2} \leq \sum_{j \in \Z} ( \sum_n |\tilde{\s}_n^* ( d_{n+j})|^2)^{1/2}, \]
defining $d_t = 0$ for $t < 0$.

Since $\tilde{\s}_n^*(g) \lesssim M g$ by previous discussion,
we may ignore the $L^2(\Z^d)$-contribution of each $|j| \leq C = C(L)$:
\[ \aligned
\sum_{|j|\leq C} \left\| \sum_n |\tilde{\s}_n^* ( d_{n+j})|^2)^{1/2} \right\|_{L^2(\Z^d)} &\lesssim
\sum_{|j| \leq C} \left\| (\sum_n |\tilde{\s}_n^* ( d_{n+j})|^2)^{1/2} \right\|_{L^2(\Z^d)} \\
&\lesssim \sum_{|j| \leq C} \left\| (\sum_n |M d_{n+j}|^2)^{1/2} \right\|_{L^2(\Z^d)} \\
&= \sum_{|j| \leq C} \left( \sum_n \left\| M d_{n+j} \right\|_{L^2(\Z^d)}^2 \right)^{1/2} \\
&\lesssim \sum_{|j| \leq C} \left( \sum_n \left\| d_{n+j} \right\|_{L^2(\Z^d)}^2 \right)^{1/2} \\
&= \sum_{|j| \leq C} \left\| (\sum_n |d_{n+j}|^2)^{1/2} \right\|_{L^2(\Z^d)} \\
&= \sum_{|j| \leq C} \| f\|_{L^2(\Z^d)} \\
&\lesssim \|f\|_{L^2(\Z^d)}, \endaligned\]
where we used the $L^2(\Z^d)$-boundedness of $M$ in the fourth line.

For $j > C$, let $Q \in \sigma_{n+j-1}$ be arbitrary, and consider $\tilde{\s}_n^* d_{n+j}(x)$ on $Q$.
Since $j > C > 0$,
\[ \tilde{\s}_n^* d_{n+j}(x) = \tilde{\s}_n^* (d_{n+j} 1_{5Q})(x),\]
we may bound
\[ |\tilde{\s}_n^* d_{n+j}(x)| \lesssim \max_{x\in 5Q} | d_{n+j} |(x) \leq (\sum_{Q_i} | d_{n+j}|^2(x_{Q_i}))^{1/2},\]
where $Q_i \in \sigma_{n+j-1}$ lie in $5Q$, and $x_{Q_i}$ is the center of each such $Q_i$; we note that $d_{n+j}$ is constant-valued on $Q \in \sigma_{n+j-1}$. Let us call this constant value $c = c(Q)$. Then, for every $x \in Q \in \sigma_{n+j-1}$ such that $x + 2H_k \subset Q$,
\[ A_t d_{n+j}(x) = c,\]
hence $\tilde{\s}_n^* d_{n+j}(x) = 0$ whenever $x + 2H_n \subset Q$.
In particular, $\tilde{\s}_n^* d_{n+j} $ is supported inside
\[ \{ x : x+2H_n \cap \partial Q \neq \emptyset\}.\]

We may consequently estimate
\[ \aligned
\sum_{Q\in \sigma_{n+j-1}} \int_Q |\tilde{\s}_n^* d_{n+j}|^2(x) &\leq \sum_{Q\in \sigma_{n+j-1}} \int \sum_{Q_i} | d_{n+j}|^2(x_{Q_i}) \cdot 1_{B(Q,2H_n)} \\
&\lesssim \sum_{Q\in \sigma_{n+j-1}} \int |d_{n+j}|^2(x) 1_{B(Q,2H_n)} \\
&= \sum_{Q\in \sigma_{n+j-1}} |d_{n+j}|^2(x_Q) \ |B(Q,2H_n)| \\
&\leq \varepsilon(j-1-2L)^2 \sum_{Q\in \sigma_{n+j-1}}|d_{n+j}|^2(x_Q) \ |Q| \\
&= \varepsilon(j-1-2L)^2 \| d_{n+j}\|_{L^2(\Z^d)}^2, \endaligned\]
where $L$ is the cost of separating scales, as in the alternative characterization of the eccentricity JRW-criterion.

Using the orthogonality of $\{d_{n+j}\}$ and summing over $n$ exhibits
\[ \left\| (\sum_n |\tilde{\s}_n^* ( d_{n+j})|^2)^{1/2} \right\|_{L^2(\Z^d)}^2 \leq \varepsilon(j-1-2L)^2 \|f\|_2^2,\]
and summing over $j > C$ shows
\[ \sum_j \left\| (\sum_n |\tilde{\s}_n^* ( d_{n+j})|^2)^{1/2} \right\|_{L^2(\Z^d)} \leq \sum_{j \geq C-2L} \varepsilon(j)^2 \|f\|_{L^2(\Z^d)} \lesssim \|f\|_{L^2(\Z^d)}.\]

Next, for $j<-C$ we establish the pointwise inequality
\[ |\tilde{\s}_n^* ( d_{n+j})|^2(x) \lesssim \iota^2(j) \cdot Md_{n+j}(x)^2,\]
where $\iota(j)$ appears as the quantitative measure of smoothness in the JRW-criterion.

Summing over $n,j$ as above will then yield the desired bound.

To do so, for $Q \in \sigma_{n+j}$, we first observe that for any $E_i \in \A_n$, any $v \in H_n$, we may bound
\[ \aligned
|A_id_{n+j}(x+v)|^2 &= \left| \frac{1}{|E_i|}\sum_{y \in x+v - E_i} d_{n+j}(y)\right|^2 \lesssim \left(\frac{1}{|H_n|} \right)^2 \left| \sum_{y \in x+v - E_i} d_{n+j}(y)\right|^2 \\
&= \left(\frac{1}{|H_n|} \right)^2 \left| \sum_{y \in x+v - E_i} \sum_{Q : Q \cap \partial(x+v-E_i)} d_{n+j}1_Q(y) \right| \\
&\leq \left(\frac{1}{|H_n|} \right)^2 |B(E_i,R_{n+j})| \cdot \sum_{y \in x+v - E_i} |d_{n+j}(y)|^2 \\
&\leq \iota^2(j) \cdot \frac{1}{|H_n|}1_{H_n}*|d_{n+j}|^2(x+v) \\
&\lesssim \iota^2(j) \cdot \frac{1}{|2H_n|}1_{2 H_n}*|d_{n+j}|^2(x) \\
&\leq \iota^2(j) \cdot  Md_{n+j}(x)^2.
 \endaligned\]
Since $\E_n d_{n+j} \equiv 0$, this immediately yields the result for the long variation.

For the short variation, the proof of \cite[Theorem B]{J1} leads to the bound
\[ |\s_{n}^S d_{n+j}(y)|^2 \lesssim \iota(j)^2 \cdot \chi_{H_k}*d_{n+j}^2(y). \]
%(For the sake of completeness, we include the details after the proof proper.)
Substituting $y= x+v$ where $v=v(x) \in H_k$ is such that
\[ \tilde{\s}_{n}^S d_{n+j}(x) = \s_{n}^S d_{n+j}(x+v)\]
yields
\[ |\tilde{\s}_{n}^S d_{n+j}(x)|^2 \lesssim \iota(j)^2 \cdot \chi_{H_k}*d_{n+j}^2(x+v) \lesssim \iota(j)^2 \cdot  \chi_{4H_k}*d_{n+j}^2(x). \]
Integrating this estimate, then summing over $n$ yields
\[ \sum_n \int |\tilde{\s}_{n}^S d_{n+j}(x)|^2 \lesssim \iota(j)^2 \cdot \sum_n \|d_{n+j}\|_{L^2(\Z^d)}^2 = \iota(j)^2 \|f\|_{L^2(\Z^d)}^2.\]
A final sum over $j < -C$ shows
\[ \sum_j \left\| (\sum_n |\tilde{\s}_n^* ( d_{n+j})|^2)^{1/2} \right\|_{L^2(\Z^d)} \leq \sum_{j < -C} \iota(j) \cdot \|f\|_{L^2(\Z^d)} \lesssim \|f\|_{L^2(\Z^d)}.\]
\end{proof}

The corollary below is a direct consequence of the previous propositions:
\begin{cor}
The discretized square functions
\[ \aligned
\s_{d}^{*}f &:= \left(\sum_k |\s_{d,k}^{*} f|^2\right)^{1/2}, \\
\s_{D}^{*}f &:= \left(\sum_k |\s_{D,k}^{*} f|^2\right)^{1/2}
\endaligned \]
$*=L,S$ are $L^2$ bounded.
\end{cor}
%Its simple structure as -- essentially -- a sum of projections allow us to use the stopping time/good-$\lambda$ machinery of \cite{W} in its study.
%In the following sections we shall use these tools to obtain quantitative bounds on the convergence.

\section{High-$L^p(\Z^d)$ Estimates}
In this section we prove Theorem \ref{hi} by way of the following
\begin{proposition}
If $\A$ is regular, then $\s_{d}^{*}f, \ *=L,S$ is $L^p(\Z^d)$-bounded, $2<p<\infty$.
\end{proposition}

We refine our (reverse) filtration $\{\sigma_k\}_k$ to $\{ \tau_{j} \}_j$, where $\tau_{j(k)} = \sigma_k$, and so successive atoms differ in size by a factor of $2$:
if
\[ R_j = \prod_{a=1}^d [0, 2^{a(j)}) \in \tau_j \]
are generating atoms, with ``symmetrized'' rectangles $H_j' = \prod_{a=1}^d [-2^{a(j)}, 2^{a(j)})$, then we have
\[ \sum_{a=1}^d |a(k+1) - a(k)| = 1.\]
We define the maximal operators
\[ \M f(x) := \sup_{j } | \E'_j f|,\]
\[ M'f(x) := \sup_{j } \sup_{x \in y+5H_j'} \frac{1}{|5H_j'|} \int_{y+5H_j'} |f| ,\]
where $\E'_j$ is the expectation with respect to $\tau_j$ (so $\E'_{j(k)} = \E_k$).
We also define the sharp function associated to our refined filtration
\[ \M^{\#} f(x): = \sup_j \sup_{x \in R \in \tau_j} \inf_{a} \frac{1}{|R|} \int_R |f - a|.\]

We certainly have that $\M^{\#} f \leq \M f$; the familiar good-$\lambda$ inequality
\[ |\{ \M f > 2\lambda, \M^{\#} f \leq \gamma \lambda \} | \lesssim \gamma |\{ \M f > \lambda \}|\]
holds (the implicit constant is in fact $2$), and by integrating distribution functions we see that the sharp function controls the maximal function in $L^p(\Z^d)$.

The key result we need is the following:
\begin{lemma}\label{sharp}
\[ \M^{\#} (\s_{d}^{*}f^2)(x) \lesssim M' (f^2)(x). \]
\end{lemma}
\begin{remark}
Informally, $\s_{d}^{*}f$ is in ``dyadic'' BMO with respect to the filtration $\{\tau_j\}$ whenever $f \in L^\infty(\Z^d)$.
\end{remark}

Assuming this lemma, with $p = 2r > 2$, we will have
\[ \aligned
\left\|\s_{d}^{*}f \right\|_{L^p(\Z^d)}^2 &= \left\| \s_{d}^{*}f^2 \right\|_{L^r(\Z^d)} \\
&\leq \left\| \M (\s_{d}^{*}f^2) \right\|_{L^r(\Z^d)} \\
&\lesssim \left\| \M^{\#} (\s_{d}^{*}f^2) \right\|_{L^r(\Z^d)} \\
&\lesssim \| M' (f^2) \|_{L^r(\Z^d)} \\
&\lesssim \| f^2 \|_{L^r(\Z^d)} \\
&= \| f\|_{L^p(\Z^d)}^2, \endaligned \]
proving the proposition.

\begin{proof}[Proof of Lemma \ref{sharp}]
Fix some $x \in \Z^d$, and let $x \in R \in \tau_j, j(k') < j \leq j(k'+1)$ be arbitrary. Express

\[ \aligned
\s_{d}^{*}f(x)^2 &= \sum_{k\leq k'} |\s_{d,k}^{*} f|^2 + \sum_{k > k'} |\s_{d,k}^{*} f|^2 \\
&= \sum_{k \leq k'} \sum_{Q \in \sigma_k} \tilde{\s}_{k}^{*} f(x_Q)^2 1_Q + \sum_{k > k'} \sum_{Q \in \sigma_k} \tilde{\s}_{k}^{*} f(x_Q)^2 1_Q \\
&= \sum_{k \leq k'} \sum_{Q \in \sigma_k} \tilde{\s}_{k}^{*} (f1_{5R}) (x_Q)^2 1_Q + c_R(f),
 \endaligned\]
by the nesting properties $R \supset Q \in \sigma_k$ for $k \leq k'$,
$R \subset Q \in \sigma_k$ for $k > k'$ for $x\in Q$.

We bound
\[ \aligned
\frac{1}{|R|}\int_R |\s_{d}^{*}f(x)^2 - c_R(f)| &=
\frac{1}{|R|} \int_R | \sum_{k \leq k'} \sum_{Q \in \sigma_k} \tilde{\s}_{k}^{*} (f1_{5R}) (x_Q)^2 1_Q | \\
&\leq \frac{1}{|R|} \int | \sum_{k \leq k'} \sum_{Q \in \sigma_k} \tilde{\s}_{k}^{*} (f1_{5R}) (x_Q)^2 1_Q | \\
&\leq \frac{1}{|R|} \left\| \s_{d}^{*}(f1_{5R}) \right\|_{L^2(\Z^d)}^2 \\
&\lesssim \frac{1}{|5R|}\int |f1_{5R}|^2 \\
&\lesssim M'(f^2)(x), \endaligned \]
which leads to
\[ \inf_{a} \frac{1}{|R|} \int_R |\s_{d}^{*}f(x)^2 - a| \leq M'(f^2)(x).\]
The lemma is proven by taking a final supremum over pertaining $R$.
\end{proof}

We now transfer this result back to $\s_D^*f$:

\begin{cor}
For each $2 < p < \infty$, $\s_D^* f$ -- the pointwise majorants of our shifted square functions -- are bounded on $L^p(\Z^d)$.
\end{cor}

\begin{proof}
Let $p = 2r > 2$, and $r'$ denote the dual exponent to $p/2 = r$. For an appropriate non-negative function $w \geq 0, \|w\|_{L^{r'}} = 1$ we estimate
\[ \aligned
\| \s_D^*f\|_p^2 &= \| |\s_D^*f|^2 \|_r\\
&= \sum_k \int |\s_{D,k}^*f|^2 w \\
&\leq \sum_k \int |\s_{D,k}^*f|^2 M_{HL,t}w, \endaligned\]
where
\[ M_{HL,t}w := ( M_{HL} (w^t) )^{1/t}\]
for $1 < t <r'$. By e.g.\ \cite[\S V.6.15]{S}, we know that $M_{HL,t} w \in A_1 \subset A_\infty$, so that we have
\[ \int |\s_{D,k}^*f|^2 M_{HL,t}w \lesssim \int |\s_{d,k}^*f|^2 M_{HL,t}w \]
for each $k$; summing appropriately yields
\[ \aligned
\| \s_D^*f\|_p^2 &\lesssim \int |\s_{d}^*f|^2 M_{HL,t}w \\
&\leq \| |\s_d^* f |^2\|_r \| M_{HL,t}w \|_{r'} \\
&= \| \s_d^* f\|_p^2 \| M_{HL,t}w \|_{r'} \\
&\lesssim \|f\|_p^2, \endaligned\]
since $M_{HL,t}$ maps $L^{r'}$ to itself any $r' > t$.
\end{proof}

As corollaries, we are now able to affirmatively answer the following problems.

\begin{problem}[\cite{J1}, Problem 7.1]
Let $D_1 \subset D_2 \subset \dots \subset \Z^d$ be a nested sequence of (closed) disks (without loss of generality containing the origin) and let $p > 2$. If
\[ A_i f (n) := \frac{1}{|D_i|} \sum_{m \in D_i} f(n-m) \]
are convolution operators, is it true that the square function
\[ \s_{d} f:= (\sum_i |(A_i -A_{i+1})f|^2)^{1/2} \]
is bounded on $L^p$ -- and thus
\[ \s_{\text{abstract}, d} := (\sum_i |(M_i -M_{i+1})f|^2)^{1/2} \]
is finite $\mu$-a.e.?
\end{problem}

\begin{problem}[\cite{J} Question 4.7] \label{Q1}
Let $\E_k f$ be the usual dyadic martingale on $[0,1)$, and let $D_kf(x)= 2^k \int_{I_k(x)} f(t) \ dt$, where $I_k(x)$ is a measurably chosen interval of length $2^{-k}$ which contains the point $x$.
Does the square function
\[ \s f(x) = ( \sum_{k=0}^\infty |D_k f(x) - \E_k f(x)|^2)^{1/2} \]
map $L^p \to L^p$ for all $2\leq p<\infty$?
\end{problem}

\section{Weighted Estimates}
\subsection{Preliminaries}
We continue to make use of the refined our (reverse) filtration $\{ \tau_{j} \}_j$, where $\tau_{j(k)} = \sigma_k$, and
the maximal operators
\[ \aligned
 \M f(x) &:= \sup_{j } | \E'_j f|, \\
 M'f(x) &:= \sup_{j } \sup_{x \in y+5H_j'} \frac{1}{|5H_j'|} \int_{y+5H_j'} |f| , \endaligned \]
were introduced in the previous section.

Throughout this section, we will assume that $M'f \lesssim M_{HL} f$ pointwise, where we recall $M_{HL}$ is the uncentered, cubic Hardy-Little maximal function.
This condition forces additional smoothness on the \emph{refined} collection $\{H_j'= \prod_{a=1}^d [-2^{a(j)}, 2^{a(j)}) \}$:
\[ \sup_{j} \max \{ |a(j) - a'(j)| : 1 \leq a,a' \leq d\} \leq K \]
for some absolute $K = O(1)$. We shall call such families $\A = \{E_i\}$ \emph{cubic}.
\begin{remark}
\emph{Cubicity} is a strictly stronger statement than the previous control over the eccentricities of the $\{ H_k\}$, as seen for instance by considering
\[ \{ H_k = [-2^{2^k}, 2^{2^k}) \times [-2^k,2^k) \} \]
inside $\Z^2$. Indeed, one may take $\varepsilon(l)^2 \lesssim 2^{-l}$, but the maximal function associated to the $\{H_k\}$ is pointwise incomparable to $M_{HLW}$.
\end{remark}

In this section, we will leverage especially strong weighted $L^2(\Z^d)$ bounds on our (sublinear) square functions to prove certain weighted weak-type $(1,1)$ estimates. The key tool used to achieve such lowering of exponents is the following classical decomposition technique due to Calder\'{o}n and Zygmund.
\begin{lemma}[Calder\'{o}n-Zygmund Decomposition]\label{CZD}
For any altitude $\lambda >0$, any $f \in L^1(\Z^d)$ may be decomposed in the form
\[f = g + b = g + \sum_P b_P, \]
where the sum runs over certain disjoint selected cubes $P \in \bigcup_j \tau_j$. $g$, the so-called ``good" function, satisfies
\begin{enumerate}
\item $\| g\|_{L^\infty(\Z^d)} \lesssim \lambda$; and
\item $\| g \|_{L^1(\Z^d)} \leq \| f \|_{L^1(\Z^d)}$.
\end{enumerate}
The ``bad" function, $b = \sum_P b_P$, satisfies the following properties:
\begin{enumerate}
\item Each $b_P$ is supported inside $P$;
\item $\int b_P = 0$;
\item $\| b_P \|_{L^1(\Z^d)} \lesssim \lambda |P|;$
\item $\sum_P |P| \leq \frac{1}{\lambda} \|f\|_{L^1(\Z^d)}$.
\end{enumerate}
\end{lemma}

With this in hand, we are ready to turn to our main results.

\subsection{Weighted Estimates}
Following the approach of \cite[\S 6]{W}, we prove the following weighted results:

\begin{proposition}[Weighted $L^2$]
For any non-negative function $w \geq 0$, there exists an absolute constant $C_2= C_2(\{\iota(l)\}, \{\varepsilon(l)\})$ so that
\[ \int |\s_{D}^{*}f|^2 w \leq C_2 \ \int |f|^2 M_{HL} w.\]
Consequently, for any $A_1$ weight $v \in A_1$, we have
\[ \int |\s_{D}^{*}f|^2 v \leq C_2 \ \int |f|^2 v.\]
\end{proposition}

\begin{proposition}[Weighted $L^{1,\infty}$]
For any $v \in A_1$, there exists an absolute constant $C_1= C_1(\{\iota(l)\}, \{\varepsilon(l)\})$ so that
\[ \lambda v( \{ \s_{D}^{*}f > \lambda \} ) \leq C_1 \ \int |f| v \]
for all $\lambda \geq 0$.
\end{proposition}

By interpolation, we get that for each $v \in A_1$
\[ \int |\s_{D}^{*}f|^p v \lesssim_p \ \int |f|^p  v \]
for all cubic families, $1<p\leq 2$. By Rubio de Francia extrapolation (cf.\ \cite[\S 7, p.143]{W} or \cite[Theorem 7.8]{D1}), we arrive at Theorem \ref{weight}:

\begin{theorem}
For all cubic families, we have
\[ \int |\s_{D}^{*}f|^p v \lesssim_p \ \int |f|^p v \]
for any $A_p$-weight $v$, $1 <p  < \infty$.
\end{theorem}

\begin{proof}[Proof of $L^2$ Estimate]
The plan is decompose our operator according to scale -- and then the size of the weight $w$:

For $j \in \Z $, we collect
\[ F(j,k) := \{ P \in \sigma_k : 2^j < \frac{w(3P)}{|3P|} \leq 2^{j+1} \} \subset \{ M_{HL} w \gtrsim 2^j\}, \]
and let $F(j) = \cup_k F(j,k)$.

Using Tonelli's theorem to interchange the order of summation, we estimate
\[ \aligned
\int |\s_{D}^{*}f|^2 w &= \sum_k \int |\s_{D,k}^*f|^2 w \\
&\lesssim \sum_k \sum_{P \in \sigma_k} \s_{k}^*f(x_P)^2 w(3P)
\\
&= \sum_j \left( \sum_{k} \sum_{P \in F(j,k)} \tilde{\s}_k^*(f)^2(x_P) w(3P) \right) \\
&\lesssim \sum_j 2^j \left( \sum_k \sum_{P \in F(j,k)} \tilde{\s}_k^*(f )^2(x_P) |P| \right). \endaligned\]

But now, for $P \in F(j)$ we know that $5P \subset \{ M_{HL} w \gtrsim_d 2^j \}$, and we therefore have the equality:
\[ \sum_k \sum_{P \in F(j,k)} \tilde{\s}_k^*(f )^2(x_P) =
\sum_k \sum_{P \in F(j,k)} \tilde{\s}_k^*(f \cdot 1_{ \{ M_{HL} w \gtrsim_d 2^j \} } )^2(x_P).\]
Consequently, using the $L^2$-boundedness of $\tilde{\s}_d$, we estimate
\[ \aligned
\sum_k \sum_{P \in F(j,k)} \tilde{\s}_k^*(f \cdot 1_{ \{ M_{HL} w \gtrsim_d 2^j \} } )^2(x_P) |P|
&\leq \left\| \s_d(f \cdot 1_{ \{ M_{HL} w \gtrsim_d 2^j \} }) \right\|_{L^2(\Z^d)}^2 \\
&\lesssim \int |f|^2 \cdot 1_{ \{ M_{HL} w \gtrsim_d 2^j \} }, \endaligned\]
and summing over $j$ shows
\[ \aligned
\int |\s_{D}^{*}f|^2 w &\lesssim \sum_j 2^j \left( \sum_k \sum_{P \in F(j,k)} \tilde{\s}_k^*(f \cdot 1_{ \{ M_{HL} w \gtrsim 2^j \} } )^2(x_P) |P| \right) \\
&\lesssim \sum_j 2^j \int |f|^2 \cdot 1_{ \{ M_{HL} w \gtrsim 2^j \} } \\
&\lesssim \int |f|^2 M_{HL}w. \endaligned\]
\end{proof}

\begin{proof}[Proof of $L^{1,\infty}$ Estimate]
We maintain our convention of assuming $f \geq 0$.
By multiplying $f$ by a suitable constant, it's enough to prove the result for $\lambda = 1$:
\[ v ( \{ \s_{D}^{*} f > 1 \}) \lesssim \int |f| v.\]

Using our Calder\'{o}n-Zygmund decomposition Lemma \ref{CZD}, we perform a Calder\'{o}n-Zygmund stopping-time decomposition at height $c = c(d,K) \lesssim 1$ depending on the dimension and the refined filtration
using $\M f$.

With $E = \{ \M f > c\}$, collect all maximal $R \in E$, $R \in \tau_j$ in $E_j$, so that we may decompose
\[ E = \bigcup_k E(k) := \bigcup_k \left( \cup_{j(k-1) < j \leq j(k)} E_j \right) .\]
We perform our Calder\'{o}n-Zygmund decomposition and arrive at
\[
f = g + b = g+ \sum_{k} b^k = g + \sum_k \left( \sum_{P \in E(k)} b_P \right), \]
where $g$ and $b$ have the above-detailed properties. This may be achieved, for instance, by setting
\[ b_P = (f - f_P)1_P,\]
 and
 \[ g = f1_{E^c} + \sum_P f_Q,\] where $f_P = \frac{1}{|P|}\int_P f$ is the average of $f$ on $P$.
%Since we have refined our initial filtration, we have the (familiar) stopping-time bound $f_Q \leq 2 c$.

Set $X = \bigcup 2^{2K} P$, where $K$ is as above. We may choose $c$ sufficiently large so that $X \subset \{ M_{HL} f > 1\}$.
We then have the standard estimate (cf. e.g.\ \cite[ \S 2.1.3]{S})
\[ |X| \lesssim \int |f| M_{HL} v \lesssim \int |f| v,\]
so our problem reduces to showing
\begin{itemize}
\item $v ( \{ X^c : \s_{D}^{*} g > 1/2 \} ) \lesssim \int |f| v$; and
\item $v ( \{ X^c : \s_{D}^{*} b > 1/2 \} ) \lesssim \int |f| v$.
\end{itemize}

For the first point, we use Chebyshev's inequality and the established $L^2(\Z^d)$ bound to majorize
\[ \aligned
v( \{ X^c: \s_{D}^{*} g > 1/2 \}) &\lesssim \int \s_{D}^{*} g^2 \ (v1_{X^c}) \\
&\lesssim \int g^2 \ M_{HL} (v1_{X^c}) \\
&\lesssim \int g \ M_{HL} (v1_{X^c}) \\
&\leq \int_{E^c} f M_{HL} v + \sum_{P} f_P \int_P M_{HL}(v1_{X^c}). \endaligned\]
Note the use of the pointwise bound $|g| \lesssim c \lesssim 1$ in passing to the third line.

But $\sup_P M_{HL} (v1_{X^c}) \lesssim \inf_P M_{HL} v$, since for any $x \in P$,
\[ M_{HL} (v1_{X^c})(x) \leq \sup_{ x \in R } \frac{1}{|R|} \int_R v,\]
where the supremum is taken over $R$ cubes, all of whose side lengths are at least as large as those of $P$. Consequently, we may estimate the above sum:
\[ \sum_{P} f_P \int_P M_{HL}(v1_{X^c}) \lesssim \sum_P (\int_P f) \cdot \inf_P M_{HL} v \leq \sum_P \int_P f M_{HL} v, \]
which yields the desired result, since $M_{HL}v \lesssim v \in A_1$.

Before beginning the second point we make the following observation: if $P$ is a selected (bad) cube, then
\[ \int_P |b| v \leq \int_P |f| v + \int_P |f_P| v \leq \int_P |f| v + |f_P| v(P).\]
But since $v \in A_1$,
\[ v(P) \lesssim |P| \inf_{x \in P} v(x),\]
so that
\[ |f_P| v(P) \lesssim \left| \int_P f \right| \cdot \inf_{x \in P} v(x) \leq \int_P |f| v.\]
Summing over all $P$, we have shown that
\[ \int|b| v \lesssim \int |f| v,\]
and so we need only show that
\[ v ( \{ X^c : \s_{D}^{*} b > 1/2 \} ) \lesssim \int |b| v.\]

With this in mind, we next note that for $R \in \sigma_k$,
\[
|\tilde{\s}_{k}^{*} b^{k-l}|(x_R)|^2 \leq \iota(l)^2 \frac{1}{|R|} \int_{x_R + 4H_{k}} |b^{k-l}|.
\]

Using this estimate, away from $X$, we majorize
\[ \aligned
|\s_{D,k}^{*} (\sum_{ l \geq 1} b^{k-l} )|^2
&\leq \left( \sum_{ l \geq 1} |\s_{D,k}^{*} (b^{k-l} )| \right)^2 \\
&\lesssim \left( \sum_{l \geq 1} \iota(l)^{1/2} \cdot  \iota(l)^{1/2} \left( \sum_{R \in \sigma_k}
\left( \int |b^{k-l}| 1_{x_R + 4H_{k}} \right) \frac{1_{3R}}{|R|} \right)^{1/2} \right)^2. \endaligned\]

We now use Cauchy-Schwartz and the summability of $\sum_l \iota(l) < \infty$ to
estimate the foregoing by a constant multiple of
\[ \sum_{l \geq 1} \iota(l) \cdot  \left( \sum_{R \in \sigma_k}
\left( \int |b^{k-l}| 1_{x_R + 4H_{k}} \right) \frac{1_{3R}}{|R|} \right).\]
%&\lesssim
%\sum_{l \geq 1} \iota(l) \cdot  \sum_{P \in \sigma_k}
%\left( \int |f 1_{E(k-l)}| 1_{P} \right) \frac{1_P}{|P|}, \endaligned \]
%where we used Cauchy-Schwartz in the third line, and the bounded overlap of $\{ x_P + 4H_k : P \in \sigma_k\}$ in the fourth.

Immediately, we have
\[ \int |\s_{D,k}^{*} (\sum_{ l \geq 1} b^{k-l} )|^2 v \lesssim
\sum_{l \geq 1} \iota(l) \cdot  \sum_{R \in \sigma_k}
\left( \int |b^{k-l}| 1_{x_R + 4H_{k}} \right) \frac{v(x_R + 4H_{k})}{|x_R + 4H_{k}|}. \]
But for any $Q \subset E(k-l)$ which intersects $x_R + 4H_{k}$,
\[ \frac{v(x_R + 4H_{k})}{|x_R + 4H_{k}|} \lesssim \inf_{x \in Q \cap (x_R + 4H_k) } M_{HL} v(x), \]
so we may replace
\[ \left( \int |b^{k-l}| 1_{x_R + 4H_{k}} \right) \frac{v(x_R + 4H_{k})}{|x_R + 4H_{k}|} \lesssim \int_{x_R + 4H_{k}} |b^{k-l}| M_{HL} v,\]
which leads directly to the bound
\[ \aligned
\int |\s_{D,k}^{*} (\sum_{ l \geq 1} b^{k-l} )|^2 v &\lesssim
\sum_{l \geq 1} \iota(l) \sum_{R \in \sigma_k} \int_{x_R + 4H_{k}} |b^{k-l}| M_{HL} v \\
&\lesssim
\sum_{l \geq 1} \iota(l) \int |b^{k-l}| M_{HL} v, \endaligned \]
due to the bounded overlap of $\{ x_R + 4H_k : R \in \sigma_k\}$.

Using Chebyshev and summing over $k,l$ completes the second point,
\[ \aligned
v( \{ X^c : \s_{D}^{*}b > 1/2 \}) &\lesssim \int | \s_{D}^{*}b|^2 ( v 1_{X^c}) \\
&\leq \sum_k \sum_{l\geq 1} \iota(l) \int |b^{k-l}| M_{HL} v \\
&= \sum_{l \geq 1} \iota(l) \int |b| M_{HL} v \\
&\lesssim \int |b| M_{HL} v, \endaligned \]
thereby concluding the proof.
\end{proof}

\section{The Rectangular Square Function}
In this section we relax our smoothness/eccentricity control over our averaging families:
suppose $\A = \{ E_i \}$ is a nested sequence of rectangles in $\Z^d$, with sides parallel to the axes, but without any regularity assumptions.

We study the following
\begin{problem}[\cite{J1} Problem 7.5]\label{?}
For each collection rectangles $\A \subset \Z^d$ set
\[ \s f = (\sum_i |A_i - A_{i+1} f|^2)^{1/2}. \]
For each $\A$, does there exist a bound
\[ \left\| \s f \right\|_{L^{1,\infty}(\Z^d)} \leq C(\A) \|f \|_{L^1(\Z^d)}? \]
\end{problem}

Following the approach of the previous sections, we study individually long/short square functions:
with $\{H_k\}$ as above,
\[ \aligned
\s_{\A}^L f &:= (\sum_{k}  | \chi_{H_k} *f - \chi_{H_{k+1}} f|^2)^{1/2}, \text{ and }\\
\s_{\A}^S f &:= (\sum_k |\s_{\A}^k f|^2)^{1/2} := (\sum_k \sum_{E_i \subset H_k, \text{ nested }} | A_{i} f - A_{i+1} f|^2)^{1/2}. \endaligned \]

Establishing the weak type bound for the short square function $\s_{\A}^S$ remains currently out of reach; we are, however, able to establish the following partial result:

\begin{proposition}\label{gl}
There exists an absolute constant independent of the collection $\A$ so that
\[ \left\| \s_{\A}^L f \right\|_{L^{1,\infty}(\Z^d)} \leq C \|f \|_{L^1(\Z^d)}.\]
\end{proposition}

This result is proven by combining the fibre-wise argument used in \cite{J2} with the Calder\'{o}n-Zygmund technique used in establishing the $L^{1,\infty}(\Z^d)$ weighted estimate.

For the sake of exposition, we pause here to record the following lemmas which will be used in the main argument below.

\begin{lemma}\label{1d1}
Suppose that
\[ \Psi_n = \sum_{J} \psi_J,\]
where the $\{\psi_J\}$ are a finite collection functions, disjointly supported in intervals $J \subset \Z$, $|J|=2^n$ and satisfy
$\int \psi_J = 0$, $\| \psi_J\|_1 \lesssim 2^n$.

If $\chi_I:= \frac{1}{|I|} 1_I$, $2^{n+s} \leq |I| < 2^{n+s+1}, \ s \geq 0$
\[ \| \chi_I * \Psi_n \|_1 \lesssim 2^{-s} \|\Psi_n\|_1.\]
\end{lemma}
\begin{proof}
If $E := \{ x : J \cap x-\partial I \neq \emptyset \}$, then $|E| \lesssim |J|$, and for each $J$,
\[ |\chi_I *\psi_J(x)| \leq \frac{\|\psi_{J} \|_1 }{|I|} 1_{E} \lesssim 2^{-s} 1_E.\]
Integrating, then summing over $J$, exhibits $\|\chi_I *\Psi_n\|_1 \lesssim 2^{-s} \| \Psi_n \|_1$, as desired.
\end{proof}

The following generalization of the Stein-Weiss Lemma \cite[Lemma 2.3]{SW} on summing weak-type inequalities will also be of use.

\begin{lemma}\label{wk}
If $\|g_k\|_{1,\infty} \leq 1$, and $c_k$ are a collection of positive constants with $\sum c_k = 1$.
Set $\gamma:= (\sum_k c_k^2)^{1/2}$.
Then
$G:=(\sum |c_k g_k|^2)^{1/2}$ has $\| G \|_{1,\infty} \lesssim 1$. By scaling:
\[ \left\| (\sum_k |g_k|^2)^{1/2} \right\|_{1,\infty} \lesssim \sum_k \|g_k\|_{1,\infty}.\]
\end{lemma}
\begin{proof}[Sketch]
This proof is similar to Stein-Weiss. One splits each
\[ g_k = l_k + m_k + u_k \]
according to size. The $l_k$ are chosen so $(\sum |l_k|^2)^{1/2} \leq \lambda/2$, one
uses a union bound to estimate the ``high'' component
\[ | \{ (\sum_k |u_k|^2)^{1/2} > \lambda/4 \} | \leq | (\sum_k |u_k|^2)^{1/2} \neq 0 |, \]
and Chebyshev to control
\[ | \{ (\sum_k |m_k|^2)^{1/2} > \lambda/4 \} | \lesssim \lambda^{-2} \sum_k \|m_k\|_2^2. \]
\end{proof}

\begin{proof}[Proof of Proposition \ref{gl}]
For notational ease, set $\s_{\A}^L = \s$.
By separating into $d$ families as in the proof of Theorem 2.1 in \cite{J2}, we may sum over consecutive rectangles which differ in the first coordinate.
We will express
\[ x = (x_1, x') \in \Z \times \Z^{d-1}, \]
and each
\[ H_k = I_k \times J_k \subset \Z \times \Z^{d-1} \]
where $J_k$ is a $d-1$-rectangle.

The set-up is quite similar to the proof of the $L^{1,\infty}(\Z^d)$ weighted estimate:

We continue to assume that $f \geq 0$, and seek to establish $| \{ \s f > 1 \}| \lesssim \|f\|_1$.

We refine our (reverse) filtration $\{ H_k\}$ as above, and consider once again the maximal operator
\[ \M f(x) := \sup_{j \geq 0} | \E_j' f|. \]
We also define the (uncentered) fibred-maximal functions
\[ M^{d-1} f(x) := \sup_k \sup_{x' \in y'+J_k} \frac{1}{|J_k|} \int_{y'+ J_k} |f|(x_1,z') \ dz',\]
and
\[ M^{1} f(x) := \sup_k \sup_{x_1 \in y_1+I_k} \frac{1}{|I_k|} \int_{y_1+ I_k} |f|(z_1,x') \ dz_1,\]
both of which enjoy the familiar one-parameter weak-type $(1,1)$ bounds.

With $E:= \{ \M f > c \}$, we collect all maximal $R \in E$, $R \in \tau_j$ in $E_j$, so that we may express
\[ E = \bigcup_k E(k) := \bigcup_k \left( \cup_{j(k-1) < j \leq j(k)} E_j \right), \]
and decompose
\[ f = g + b = g + \sum_k b^k = g + \sum_k \left( \sum_{Q \in E(k)} b_Q \right) ,\]
where $g:= f \cdot 1_{E} + \sum_Q f_Q \cdot 1_Q$, and $b_Q = (f-f_Q) \cdot 1_Q$ as above.

With $X:= \bigcup_Q 10Q$, so that $Q + H_k \subset X$ whenever $Q \in E(k)$,
 we have the estimate
\[ |X| \lesssim \sum_Q |Q| \lesssim \sum_Q \int_Q f \leq \|f\|_1,\]
and using the $L^2(\Z^d)$-boundedness of $\s$ (\cite{J2}, Theorem 2.1) estimate
\[ | \{ \s g > 1/2 \} | \lesssim \| \s g\|_{L^2(\Z^d)}^2 \lesssim \|g \|_{L^2(\Z^d)}^2 \lesssim \|g\|_{L^1(\Z^d)} \leq \|f\|_{L^1(\Z^d)};\]
the argument reduces to showing
\[ | \{ X^c : \s b > 1/2 \} | \lesssim \|f\|_{L^1(\Z^d)}.\]

To do this, we further decompose each
\[ b_Q = b^0_{Q} + b^1_{Q} \]
where each $b^i_Q$ satisfies the size condition $\|b^i_Q\|_{L^1(\Z^d)} \sim |Q|$, but also the more specialized ``fibred'' moment conditions:
\[ \aligned
\int_{\Z} b_Q^0(x_1,x') dx_1 &= 0 \; \text{ for each } x' \in \Z^{d-1}, \\
\int_{\Z^{d-1}} b_Q^1(x_1,x') dx' &= 0 \; \text{ for each } x_1 \in \Z. \endaligned\]
This may be accomplished for instance by setting
\[ \aligned
b_Q &= b_Q^0 + b_Q^1 \\
&= \left( b_Q - \left( \int_{I_Q} b_Q(x_1,x') \ dx_1 \right) \cdot 1_{I_Q} \right) + \left( \left (\int_{I_Q} b_Q(x_1,x') \ dx_1 \right) \cdot 1_{I_Q} \right)
\\
&:=\psi_{I_Q} \otimes 1_{J_Q} + 1_{I_Q} \otimes \psi_{J_Q}, \endaligned \]
where
\[ \|\psi_{I_Q} \otimes 1_{J_Q}\|_{L^1(\Z^d)} + \|1_{I_Q} \otimes \psi_{J_Q}\|_{L^1(\Z^d)} \lesssim \|b_Q\|_{L^1(\Z^d)}. \]
We define
\[ b^k = b^k_0 + b^k_1 \]
in the obvious way.

It suffices to separately estimate
\[ | \{ X^c : \s b^0 > 1/4 \} |, \; |\{ X^c : \s b^1 >1/4 \} |.\]
We begin by studying the square function's behavior on $b^0$ -- i.e. we assume that each
\[ b_Q^0 = \psi_{I_Q} \otimes 1_{J_Q}\]
 has mean-zero when integrated in the $x_1$-direction.

We decompose the convolution operators
\[ \aligned
(\chi_{H_k} - \chi_{H_{k+1}})*f &= \chi_{J_k}^2 * ( ( \chi_{I_k}^1 - \chi_{I_{k+1}}^1 )*f ) \\
&=
( \chi_{I_k}^1 - \chi_{I_{k+1}}^1 )* (\chi_{J_k}^2 * f )
 \\
&=: \eta_k^1 * ( \chi_{J_k}^2 * f) \\
&= \chi_{J_k}^2 * (\eta_k^1 *f), \endaligned \]
with convolution involving $\chi_{J_k}^2$ taking place in the final $d-1$ coordinates, and $\chi_{I_k}^1, \chi_{I_{k+1}}^1$ in the first coordinate.
This allows us to bound -- on $X^c$ --
\[  \aligned
\s b^0 &\leq  \left(\sum_k \chi_{J_k}^2 * |\eta_k^1*\left( \sum_{|I_Q| \leq |I_k| } b_Q^{0} \right)|^2 \right)^{1/2}  \\
&\leq
\sum_{s \geq 0} \left(\sum_k \chi_{J_k}^2 * |\eta_k^1* \left( \sum_{|I_Q| = 2^{-s}|I_k|} b_Q^{0} \right)|^2 \right)^{1/2}. \endaligned\]

With $s \geq 0$ fixed, we use the Fefferman-Stein vector-valued maximal inequality (cf. e.g.\ \cite{S}, \S 2 Theorem 1) to estimate
\begin{align*}
 &\left\| \left(\sum_k \chi_{J_k}^2 * \left|\eta_k^1* \left( \sum_{|I_Q| = 2^{-s}|I_k|} b_Q^{0} \right) \right|^2 \right)^{1/2} \right\|_{L^{1,\infty}(\Z^d)}
\\
&\qquad \qquad \qquad \leq
\left\| \left(\sum_k M^{d-1} \left|\eta_k^1* \left( \sum_{|I_Q| = 2^{-s}|I_k|}  b_Q^{0} \right) \right|^2 \right)^{1/2} \right\|_{L^{1,\infty}(\Z^d)}  \\
&\qquad \qquad \qquad \lesssim \left\| \left(\sum_k \left|\eta_k^1* \left( \sum_{|I_Q| = 2^{-s}|I_k|} b_Q^{0} \right) \right|^2 \right)^{1/2} \right\|_{L^{1}(\Z^d)} \\
&\qquad \qquad \qquad \leq \sum_k \left\|\eta_k^1* \left( \sum_{|I_Q| = 2^{-s}|I_k|} b_Q^{0} \right) \right\|_{L^{1}(\Z^d)}.
\end{align*}

By applying Lemma \ref{1d1} to the functions
\[ x_1 \mapsto \chi_{I_k} * \left(\sum_{|I_Q| = 2^{-s}|I_k|} b_Q^{0}(x_1,x') \right), \]
integrating first in $x_1\in \Z$, then over $x' \in \Z^{d-1}$,
we estimate this final sum by a constant multiple of
\[ \sum_k 2^{-s} \left\| \sum_{|I_Q| = 2^{-s}|I_k|} b_Q^{0} \right\|_1 \leq 2^{-s} \| b^0 \|_{L^{1}(\Z^d)} \lesssim 2^{-s} \|f\|_{L^{1}(\Z^d)}. \]
Combining our estimates in $s \geq 0$ and choosing $c > 0$ appropriately small leads to the desired bound:
\[ \aligned
| \{ \s b^0 > 1/4 \} | &\leq \sum_{s \geq 0} \left| \left \{ \left( \sum_k \chi_{J_k}^2 * \left|\eta_k^1* \left( \sum_{|I_Q| = 2^{-s}|I_k|} b_Q^{0} \right) \right|^2 \right)^{1/2} > c 2^{-s/2} \right\} \right| \\
&\lesssim \sum_{s \geq 0 } 2^{s/2} \left\| \left( \sum_k \left|\eta_k^1* \left( \sum_{|I_Q| = 2^{-s}|I_k|} b_Q^{0} \right)\right|^2 \right)^{1/2} \right\|_{L^{1}(\Z^d)} \\
&\lesssim \sum_{ s \geq 0 } 2^{s/2} \cdot 2^{-s} \|f \|_{L^{1}(\Z^d)} \\
&\lesssim \|f \|_{L^{1}(\Z^d)}. \endaligned \]

%\[ b = \sum_Q b_Q = \sum_m b^m, \text{ where } b^m:= \sum_{Q : |J_Q| = m} b_Q$.

%First, suppose that $b_Q = 1_{I_Q}\otimes \psi_{J_Q}$, where $\psi_{J_Q}$ is supported in $J_Q$, and satisfies the size and moment conditions:
%\[ \int_{\Z^{d-1}} |\psi_{J_Q}|(y) \ dy \approx |J_Q|,  \ \int_{\Z^{d-1}} \psi_{J_Q}(y) = 0.\]

In passing to the second case, where we must estimate
\[ | \{ X^c : \s(b^1) > 1/4 \}|, \]
we collect cubes according to separation of scales of the final $(d-1)$-coordinates.

For $Q = I_Q \times J_Q \in \tau_j, j< j(k)$, we use $\triangle(Q,H_k)= \triangle_1(Q,H_k)$ to denote the degree of separation between $J_Q,J_k$.
%\[ \{ Q : \triangle(Q,H_k) = \triangle_1(Q,H_k) =s \} \]
%denote the $\{Q\}$ so that $J_Q$ is $s$-separated from $J_k$ (i.e.\ we use $\triangle(Q,H_k)$ to denote the degree of separation between $J_Q,J_k$).

Now, away from $X = \bigcup_Q 10Q$, we use the triangle inequality to obtain the pointwise bound
\[ \aligned
\s b^1 &\leq \left( \sum_{k} \left|(\chi_{H_k} - \chi_{H_{k+1}} ) *
\left(\sum_{s \geq 0} \sum_{\triangle(Q,H_k) = s} b^1_Q  \right) \right|^2\right)^{1/2} \\
&\leq \sum_{s \geq 0} \left( \sum_{k} \left|(\chi_{H_k} - \chi_{H_{k+1}} ) * \left( \sum_{\triangle(Q,H_k) = s} b^1_Q \right) \right|^2 \right)^{1/2} \\
&= \sum_{s \geq 0} (\sum_{i} |h_{J_{k_i},s}|^2)^{1/2},   \endaligned \]
where
\[
h_{J_{k_i},s}:= \left(\sum_{H_k: J_k = J_{k_i}}
\left|(\chi_{H_k} - \chi_{H_{k+1}} ) * \left(\sum_{\triangle(Q,H_{k_i}) = s} b^1_Q \right) \right|^2 \right)^{1/2},
\]
with the sum taken over all rectangles $\{H_k\}$ which share the same final $d-1$-dimensions.

We will show that for each $s$,
\[ \| h_{J_{k_i},s} \|_{L^{1,\infty}(\Z^d)} \lesssim 2^{-s} \left\| \sum_{\triangle(Q,H_{k_i}) = s} b^1_Q \right\|_{L^{1}(\Z^d)},\]
then will sum on $J_i$ and apply Lemma \ref{wk} to conclude that
\[ \aligned
\left\| (\sum_{i} |h_{J_{k_i},s}|^2)^{1/2} \right\|_{L^{1,\infty}(\Z^d)} &\lesssim \sum_{i} \| h_{J_{k_i},s} \|_{L^{1,\infty}(\Z^d)}
\\
&\lesssim 2^{-s} \sum_{i} \left\| \sum_{\triangle(Q,H_{k_i}) = s} b^1_Q \right\|_{L^{1}(\Z^d)} \\
&\leq 2^{-s} \|b^1\|_{L^{1}(\Z^d)} \\
&\lesssim 2^{-s} \|f\|_{L^{1}(\Z^d)}, \endaligned\]
so that, once again, with $c> 0$ an absolute constant, we estimate
\[ \aligned
| \{ X^c : \s b^1 > 1/4 \} | &\leq \sum_{s\geq 0} | \{ (\sum_{i} |h_{J_{k_i},s}|^2)^{1/2} > c 2^{-s/2} \}| \\
&\lesssim \sum_{s\geq 0} 2^{s/2} \cdot 2^{-s} \|f\|_{L^{1}(\Z^d)} \\
&\lesssim \|f\|_{L^{1}(\Z^d)}. \endaligned\]

To do this, we express
\[ h_{J_{k_i},s} = \left(\sum_{H_k: J_k = J_{k_i}} \left| \eta_k^1* \left( \chi_{J_{k_i}} * \left( \sum_{\triangle(Q,H_{k_i}) = s} b^1_Q \right) \right) \right|^2 \right)^{1/2} \]
as a one-dimensional square function applied to the function
\[ \chi_{J_{k_i}} * \left(\sum_{\triangle(Q,H_{k_i}) = s} b^1_Q \right),\]
which has small $L^1(\Z^d)$ norm by the separation of scales Lemma \ref{1d1}. In particular, by considering the functions
\[ x' \mapsto \chi_{J_{k_i}} * \left( \sum_{\triangle(Q,H_{k_i}) = s} b^1_Q(x_1,x') \right) \]
we estimate
\[ \aligned
\left\| \chi_{J_{k_i}} * \left( \sum_{\triangle(Q,H_{k_i}) = s} b^1_Q(x_1,x') \right) \right\|_{l^1(\Z^d)} &=
\int_{\Z} \ \left\| \chi_{J_{k_i}} * \left( \sum_{\triangle(Q,H_{k_i}) = s} b^1_Q(x_1,-) \right) \right\|_{l^1(\Z^{d-1})} \ dx_1 \\
&\lesssim \int_{\Z} \ 2^{-s} \left\| \sum_{\triangle(Q,H_{k_i}) = s} b^1_Q(x_1,-) \right\|_{l^1(\Z^{d-1})} \ dx_1 \\
&= 2^{-s} \left\| \sum_{\triangle(Q,H_{k_i}) = s} b^1_Q \right\|_{L^{1}(\Z^d)}. \endaligned \]

But now, using the one-dimensional square function result of \cite{J}, for any $\lambda \geq 0$ we may bound
\[ \aligned
 &\lambda | \{ (x_1,x') \in \Z \times \Z^{d-1}: h_{J_{k_i},s}(x_1,x') > \lambda \} | \\
 &\qquad \qquad \qquad =
\int_{\Z^{d-1}} \lambda| \{ x_1 \in \Z: h_{J_i}(-,x') > \lambda \} | \ dx' \\
&\qquad \qquad \qquad \lesssim \int_{\Z^{d-1}} \left( \int_{\Z} \left| \chi_{J_{k_i}} * \left( \sum_{\triangle(Q,H_{k_i}) = s} b^1_Q \right) (x_1,x') \right| \ dx_1 \right) \ dx' \\
&\qquad \qquad \qquad = \left\|\chi_{J_{k_i}} * \left( \sum_{\triangle(Q,H_{k_i}) = s} b^1_Q \right) \right\|_{l^1(\Z^d)} \\
&\qquad \qquad \qquad \leq 2^{-s} \left\|\sum_{\triangle(Q,H_{k_i}) = s} b^1_Q \right\|_{l^1(\Z^d)}, \endaligned \]
as desired.
\end{proof}

\section{Appendix: A Weak-Type Principle for Square Functions}
With $T$ a free $\Z^d$ action,
\[ T_y f(x):= f(T_{-y} x), \]
on a non-atomic probability space $(X,\Sigma,\mu)$,
we consider the square function:
\[ Sf(x) = S^{\tau}_{(E_m)} f(x):= (\sum_m |(M_m - M_{m+1})f|^2)^{1/2}(x) :=  (\sum_m |K_mf|^2)^{1/2}(x).\]
Note that by the measure-preserving action of $T$ and the triangle inequality
\[ \|K_m f \|_{L^1(X)} \leq 2 \|f\|_{L^1(X)},\]
i.e.\ the $\{K_m\}$ are uniformly bounded in operator norm.

\begin{theorem}\label{red}
We have the following equivalence
\begin{itemize}
\item For each $f \in L^1(X,\Sigma, \mu)$, $Sf < \infty$ $\mu$-a.e.; and
\item $ \left\| S f \right\|_{L^{1,\infty}(X)} \lesssim \|f\|_{L^1(X)}$.
\end{itemize}
\end{theorem}
\begin{remark}
The only key property about the integrability class $L^1(X)$ used in the below proof is that $1 \leq 2$;
the above equivalence persists for all $1 \leq p \leq 2$.
\end{remark}

That the second point implies the first is clear, so we concern ourselves with the remaining implication. We actually will prove (a strengthened version of) the contrapositive, namely that if no bound on the operator norm $S: L^1(X) \to L^{1,\infty}(X)$ exists, then there exists an $f \in L^1(X)$ so that $Sf = +\infty$ a.e.

The argument below is a similar to that of \cite{S1}, though to the best of our knowledge has yet to appear in print. In the spirit of the Conze principle \cite{Conze}, we reduce to the ergodic case:

\begin{proposition}
Suppose that either (and hence both) of the conditions in Theorem \ref{red} hold for some (free) system. Then the same is true of every such system.

In particular:
$\| S \phi \|_{L^{1,\infty}(X)} \leq C \|\phi\|_{L^1(X)}$ if and only if
$\| \s f \|_{L^{1,\infty}(\Z^d)} \leq C \|f\|_{L^1(\Z^d)}$.
\end{proposition}
The forward implication is established by the
Rokhlin Lemma \cite{OW}; it is here that we make use of the freeness of our action $T$. The reverse implication is a direct application of Calder\'{o}n's transference principle \cite{C1}.

For our purposes, we will henceforth assume that our action $T$ is ergodic.

\subsection{Preliminaries}
The following results will be of use.

\begin{lemma}[Randomization Lemma]\label{RL}
If $\{E_n\} \subset X$ have $\sum_n \mu(E_n) = +\infty$, then there exist a collection of vectors $y(n) \subset \Z^d$ so that
\[ \limsup T_{-y(n)} E_n = X \ \mu-a.e.\]
\end{lemma}
In the probabilistic setting, the Borel-Cantelli lemma says that if the $\{E_n\}$ are independent events with
$\sum_n \mu(E_n) = \infty$, then $\limsup E_n = X$ almost surely. The content of this lemma is that the ergodicity of the $T$-action is sufficiently randomizing to force similar independence.
\begin{proof}
If $Q(N):= \{ y \in \Z^d : |y(i)| \leq N, 1 \leq i \leq d\}$, by the ergodicity of $T$,
we know that for any measurable $A, B \subset X$
\[ \frac{1}{(2N+1)^d} \sum_{y \in Q(N)} \mu(T_{-y}A \cap B) \to \mu(A)\mu(B).\]
Consequently, we may find a $y = y(A,B)$ so that
\[ \mu(T_{-y}A \cap B) \geq 1/2 \mu(A)\mu(B);\]
it is this point which anchors the volume-packing argument found e.g.\ in \cite[\S 10]{S}.
\end{proof}

We also recall the following orthogonality lemma concerning the well-known Rademacher functions, $\{r_m(s)\}$:
\begin{lemma}[\cite{Z}, \S V, Theorem 8.2]\label{IND}
If $\sum_m |a_m|^2 < \infty$, then for (Lebesgue) almost every $s \in [0,1]$, $|\sum_m r_m(s) a_m|$ converges.
Conversely, if $\sum_m |a_m|^2 = +\infty$, then $|\sum_m r_m(s) a_m|$ diverges a.e.
\end{lemma}

For what is to follow, we shall need the following two-dimensional variant.

\begin{lemma}[Two-Dimensional Orthogonality Lemma]\label{ORT}
Suppose that $E \subset [0,1]^2$ is non-null, and assume that
\[ G(t)^2:= \lim_{M} \lim_N |\sum_{m \leq M} \sum_{n \leq N} r_m(s)r_n(t) a_{mn}|^2 \]
satisfies $|G(t)| < A$ on $E$.

Then, there exists a $M$ so that
\[ \sum_{m >M} \sum_{n} |a_{mn}|^2 \leq CA^2\]
for some absolute constant $C$.
\end{lemma}
\begin{proof}
Set $\gamma_{mn}(s,t):= r_m(s)r_n(t)$, and note that since
\[\{ \gamma_{mn} \gamma_{m'n'} : (m,n) \neq (m',n') \}\]
 are orthonormal over $[0,1]^2$,
\[ \sum_{(m,n) \neq (m',n')} |\langle \gamma_{mn} \gamma_{m'n'} | 1_E \rangle|^2 \leq |E| \]
converges absolutely. Consequently, we may pick an $M_0$ sufficiently large so that for all $M \geq M_0$,
\[ \sum_{(m,n) \neq (m',n'), m,m' > M} |\langle \gamma_{mn} \gamma_{m'n'}, 1_E \rangle|^2 < (|E|/2)^2. \]

Possibly after increasing $M_0$, we may also assume that for all $M \geq M_0$
\[ \limsup_N |\sum_{m \leq M} \sum_{n \leq N} r_m(s)r_n(t) a_{mn}| < A.\]
%for each such $M$, there exists an $N(M)$ so large that for all $N \geq N(M)$,
%\[ |\sum_{m \leq M} \sum_{n \leq N} r_m(s)r_n(t) a_{mn}| < A\]
%as well.

Now, for any $M > M' \geq M_0$
\[ \aligned
&\limsup_{N \to \infty} |\sum_{M' < m \leq M} \sum_{n \leq N} r_m(s)r_n(t) a_{mn}|^2 \\
& \qquad \qquad \qquad \lesssim
\limsup_{N \to \infty} |\sum_{m \leq M'} \sum_{n \leq N} r_m(s)r_n(t) a_{mn}|^2 +
\limsup_{N \to \infty} |\sum_{m \leq M} \sum_{n \leq N} r_m(s)r_n(t) a_{mn}|^2 \\
& \qquad \qquad \qquad \lesssim A^2 + A^2, \endaligned\]
so we may use dominated convergence to estimate

\[ \aligned
&\lim_{M \to \infty} \lim_{N \to \infty} \int_E |\sum_{M' < m \leq M} \sum_{n \leq N} r_m(s)r_n(t) a_{mn}|^2 \\
& \qquad \qquad \qquad =
\lim_{M \to \infty} \int_E \lim_{N \to \infty} |\sum_{M' < m \leq M} \sum_{n \leq N} r_m(s)r_n(t) a_{mn}|^2 \\
& \qquad \qquad \qquad = \int_E \lim_{M \to \infty} \lim_{N \to \infty} |\sum_{M' < m \leq M} \sum_{n \leq N} r_m(s)r_n(t) a_{mn}|^2 \\
& \qquad \qquad \qquad \lesssim \int_E |G(t)|^2 + \int \limsup_{N \to \infty} |\sum_{m \leq M'} \sum_{n \leq N} r_m(s)r_n(t) a_{mn}|^2 \\
& \qquad \qquad \qquad \lesssim A^2 |E| + A^2 |E| \\
& \qquad \qquad \qquad \lesssim A^2 |E|. \endaligned \]

%\[ \aligned
%&\int_E \lim_{M \to \infty} \lim_{N \to \infty} |\sum_{M' < m \leq M} \sum_{n \leq N} r_m(s)r_n(t) a_{mn}|^2 \lesssim \\
%&\int_E |G(t)|^2 + \int \lim_{N \to \infty} |\sum_{m \leq M'} \sum_{n \leq N} r_m(s)r_n(t) a_{mn}|^2 \\
%&\lesssim A^2 |E| + A^2 |E| \lesssim A^2 |E|. \endaligned \]

%On the other hand, we may use dominated convergence twice to estimate
%\[ \aligned
%&\int_E \lim_{M \to \infty} \lim_{N \to \infty} |\sum_{M' < m \leq M} \sum_{n \leq N} r_m(s)r_n(t) a_{mn}|^2 = \\
%&\lim_{M \to \infty} \int_E \lim_{N \to \infty} |\sum_{M' < m \leq M} \sum_{n \leq N} r_m(s)r_n(t) a_{mn}|^2 = \\
%&\lim_{M \to \infty} \lim_{N \to \infty} \int_E |\sum_{M' < m \leq M} \sum_{n \leq N} r_m(s)r_n(t) a_{mn}|^2; \endaligned \]
%since for $M,M' \geq M_1$:
%\[ \aligned
%&\limsup_{N \to \infty} |\sum_{M' < m \leq M} \sum_{n \leq N} r_m(s)r_n(t) a_{mn}|^2 \lesssim \\
%&\limsup_{N \to \infty} |\sum_{m \leq M'} \sum_{n \leq N} r_m(s)r_n(t) a_{mn}|^2 +
%\limsup_{N \to \infty} |\sum_{m \leq M} \sum_{n \leq N} r_m(s)r_n(t) a_{mn}|^2 \lesssim \\
%&A^2 + A^2. \endaligned \]

We now seek a lower bound on the $\lim_M \lim_N \int_E |\sum_{M' < m \leq M} \sum_{n \leq N} r_m(s)r_n(t) a_{mn}|^2$.
To do this, we expand the square
\[ \int_E |\sum_{M' < m \leq M} \sum_{n \leq N} r_m(s)r_n(t) a_{mn}|^2 \]
to get
\[ \aligned
\int_E \sum_{M' < m \leq M} \sum_{n \leq N} |a_{mn}|^2 &+ \int_E \sum_{M' < m\neq m' \leq M} \sum_{n\neq n' \leq N} \gamma_{mn}\gamma_{m'n'} a_{mn} \bar{a}_{m'n'}\\
&=: I(M,N) + II(M,N). \endaligned \]

We will show that
\[ |II(M,N)| < \frac{1}{2} I(M,N) = \frac{|E|}{2} \sum_{M' < m \leq M} \sum_{n \leq N} |a_{mn}|^2.\]
This will allow us to bound from below:
\[ \aligned
&\lim_{M \to \infty} \lim_{N \to \infty} \int_E |\sum_{M' < m \leq M} \sum_{n \leq N} r_m(s)r_n(t) a_{mn}|^2 \\
& \qquad  \qquad \qquad \geq \lim_{M \to \infty} \lim_{N \to \infty} \frac{|E|}{2} \sum_{M' < m \leq M} \sum_{n \leq N} |a_{mn}|^2 \\
& \qquad \qquad \qquad = \frac{|E|}{2} \sum_{M' < m } \sum_{n } |a_{mn}|^2. \endaligned\]
Combining this with our upper bound will allow us to conclude:
\[ \sum_{M' < m } \sum_{n } |a_{mn}|^2 \lesssim A^2.\]

But we may use Cauchy-Schwarz to majorize:
\[ \aligned
&\left| \int_E \sum_{M' < m\neq m' \leq M} \sum_{n\neq n' \leq N} \gamma_{mn}\gamma_{m'n'} a_{mn} \bar{a}_{m'n'} \right| \\
& \qquad \leq  \sum_{M' < m\neq m' \leq M} \sum_{n\neq n' \leq N} | \langle \gamma_{mn}\gamma_{m'n'} | 1_E \rangle| |a_{mn} a_{m'n'} | \\
& \qquad \leq (\sum_{M' < m\neq m' \leq M} \sum_{n\neq n' \leq N} | \langle \gamma_{mn}\gamma_{m'n'} | 1_E \rangle|^2 )^{1/2} \cdot
(\sum_{M' < m\neq m' \leq M} \sum_{n\neq n' \leq N} |a_{mn} a_{m'n'} |^2 )^{1/2} \\
& \qquad \leq \frac{|E|}{2} \cdot \sum_{M'<m \leq M} \sum_{n \leq N} |a_{mn}|^2. \endaligned\]
\end{proof}

We now turn to the proof proper.

\begin{proof}[Proof of the Theorem \ref{red}]
We proceed as in \cite{S1}:

Suppose that no bound $\|Sf\|_{L^{1,\infty}(X)} \leq N \|f\|_{L^1(X)}$ existed. In this case we could find a sequence of functions
$\{g_n\} \in L^1(X)$, and a monotonically increasing sequence $\{R_n\}$ so
%In this case we could find a sequence $h_j \subset X$ and positive numbers $\{a_j\}$ so that
%\[ \mu( Sh_j > a_j ) > \frac{2^{2j}}{a_j} \|h_j\|_1 = 2^j \cdot \| \frac{2^{j}}{a_j} \cdot h_j\|_1. \]
%If we set $g_j := \frac{2^j}{a_j} h_j$, the above inequality becomes
%\[ \mu( S g_j > 2^j) > 2^j \|g_j\|_1.\]
%We now extract a (possibly very repetitive) subcollection $\{ g'_n\}$ by repeating each $g_j$
%\[ C_j \approx \mu(S g_j > 2^j)\]
%times, and a pertaining positive, monotonically increasing sequence $\{R_n\}$, with $R_n = 2^j$ if $g'_n= g_j$.
%Then:
\[ \sum_n \mu(Sg_n > R_n) = \infty \]
diverges, while
\[ \sum_n \|g_n\|_1 \lesssim 1 \]
converges.

We apply our Randomization Lemma \ref{RL} to the collection of sets
\[ E_n:=\{Sg_n > R_n\}\]
so that
\[ X_0:=\limsup T_{-y(n)} E_n \subset X\]
has full measure, and set $f_n:= T_{y(n)}g_n.$

Then $Sf_n > R_n \to \infty$ on $X_0$, while $\sum_n \|f_n\|_1 = \sum_n \|T_{y(n)} g_n\|_1 < \infty$.

We consider the formal sum
\[ \sum_n r_n(t) f_n(x); \]
as in \cite{S1} we may find a subsequence $\{N_k\}$ along which $F(x,t):= \lim_k \sum_{n \leq N_k} r_n(t) f_n(x)$ satisfies
%We will show that $F$ converges in the appropriate sense, and that
\begin{enumerate}
\item For almost every $t \in [0,1]$, $F(x,t) \in L^1(X)$;
\item For $\mu$-a.e.\ $x \in X$, for each $m$
\[ K_m F(x,t) \equiv \lim_{k} \sum_{n \leq N_k} r_n(t) K_m f_n(x) \]
as functions of $t$
\end{enumerate}

We will prove:
\[ \text{ for almost every $t$, $SF(x,t) = + \infty$ $\mu$-a.e. }\]

To do this, we proceed by contradiction, and assume that the sum
\[ \sum_m |K_m F(x,t)|^2 < \infty \]
converges on a set $D \subset X \times [0,1]$ of positive product measure.

By Lemma \ref{IND}, for each $(x,t) \in D$
\[ | \sum_m r_m(s) K_m F(x,t) |  < \infty \]
$s$-a.e.
 so we may extract a subset $E \subset X \times [0,1]^2$, $A > 0$, so that for each $M \geq M'$ sufficiently large partial summand
\[ | \sum_{m \leq M} r_m(s) K_m F(x,t) | < A \]
on $E$.

For each measurable section ($\mu$-almost all sections are measurable), we set
\[ E_x := \{ (s,t) \in [0,1]^2: (x,s,t) \in E \} \subset [0,1].\]
Since $E$ has positive product measure, we may extract some $\delta > 0$ and a set $X_\delta$ of positive $\mu$-measure so that for each
$x \in X_\delta$, $|E_x| \geq \delta$.

For each $x \in X_\delta$, we apply our Two-Dimensional Orthogonality Lemma \ref{ORT} to
\[ \lim_{M \to \infty} | \sum_{m \leq M} r_m(s) K_m F(x,t) | =
\lim_{M \to \infty} \lim_{N_k \to \infty} | \sum_{m \leq M} \sum_{n \leq N_k} r_m(s)r_n(t) K_m f_n(x) |, \]
with $K_mf_n(x)$ in the role of $a_{mn}$.

In particular, we extract $M(x)$ so that
\[ \sum_{M(x) < m} \sum_{n} |K_m f_n(x)|^2 \leq C A^2.\]
Assume that $M(x)$ is minimal subject to the condition that
\[ \sum_{M(x) < m} \sum_{n} |K_m f_n(x)|^2 \leq 2 C A^2,\]
and collect
\[ A_k:= \{ x \in X_\delta: M(x) = k\}, \]
so that we may express $X_\delta = \bigcup_{k \geq 1} A_k$ as a disjoint union.
Choose $k'$ as small as possible subject to the constraint that $\mu(A_{k'}) > 0$. % (such a $k'$ must exist since $\mu(X_\delta) > 0$).
We have:
\[  \sum_{k' < m} \sum_n |K_m f_n(x)|^2 \lesssim A^2 \]
on $A_{k'}$.

We now wish to show that
\[ \aligned
& \left\{ y : \sum_{m \leq k'} \sum_n |K_mf_n(y)| \geq \frac{R_n}{2} \text{ for all but finitely many } n \right\} \\
&\qquad \qquad \qquad =\liminf_n \left\{ y : \sum_{m \leq k'} \sum_n |K_mf_n(y)| \geq \frac{R_n}{2} \right\}  \endaligned \]
is $\mu$-null. This will allow us to conclude that for almost every $x \in A_{k'} \cap X_0$ -- and thus almost every $x \in A_{k'}$, since $\mu(X_0)=1$ --
there exist infinitely many $n$ with
\[ \sum_{k' < m} |K_m f_n(x)|^2 \geq |Sf_n(x)|^2 -  \sum_{m \leq k' } |K_m f_n(x)|^2  \geq \frac{3}{4} R_n^2.\]
Summing in $n$ would then force
\[ \sum_{k' < m} \sum_n |K_m f_n(x)|^2 \]
to diverge, contradicting the upper bound of $\lesssim A^2$, and concluding the argument.

To this end, for each $n$ we estimate
\[ \sum_{m \leq k'} \|K_m f_n(x)\|_1 \lesssim k' \|f_n\|_1,\]
from which it follows that
\[ \mu \ \left( \left\{ y \in X:  \sum_{m \leq k'} |K_m f_n|(y) > R_n/2 \right\} \right) \lesssim k' \cdot \frac{\|f_n\|_1}{R_n}.\]
Consequently, for each $l$, we may estimate
\[ \mu \left( \bigcap_{n\geq l} \left \{ y \in X:  \sum_{m \leq k'} |K_m f_n|(y) > R_n/2 \right\} \right)  \lesssim \lim_{n \to \infty} k' \cdot \frac{\|f_n\|_1}{R_n} = 0,\]
which yields
\[ \mu  \ \left( \liminf_n \left\{ y \in X:  \sum_{m \leq k'} |K_m f_n|(y) > R_n/2 \right\} \right) = 0, \]
as desired.
\end{proof}

% Put bibliography items here

\end{document}